%% file: spectrum.tex
\documentclass[11pt]{amsart}
\usepackage[utf8]{inputenc}

\usepackage{a4wide}

\usepackage{amsfonts}
\usepackage{amsmath}
\usepackage{amsthm}
\usepackage{amssymb}
\usepackage{graphicx}
\usepackage{float}
\usepackage{enumerate}

\usepackage{thmtools}
\usepackage{thm-restate}

\usepackage{tikz}
\usetikzlibrary{decorations.markings}
\usepackage{ifthen}

\newtheorem{thm}{Theorem}[section]
\newtheorem{prop}[thm]{Proposition}
\newtheorem{lem}[thm]{Lemma}
\newtheorem{cor}[thm]{Corollary}

\theoremstyle{definition}
\newtheorem{dfn}[thm]{Definition}
\newtheorem{ex}[thm]{Example}
\newtheorem*{ex*}{Example}
\newtheorem{rmk}[thm]{Remark}

\DeclareMathOperator{\Aut}{Aut}
\DeclareMathOperator{\Epi}{Epi}
\DeclareMathOperator{\Stab}{Stab}
\DeclareMathOperator{\Ker}{Ker}

\DeclareMathOperator{\Sch}{Sch}

\DeclareMathOperator{\spec}{sp}
\DeclareMathOperator{\supp}{supp}

\newcommand{\N}{\mathbb{N}}
\newcommand{\Z}{\mathbb{Z}}
\newcommand{\R}{\mathbb{R}}
\newcommand{\C}{\mathbb{C}}

\newcommand{\F}{\mathcal{F}}
\newcommand{\CH}{\mathcal{H}}
\newcommand{\M}{\tilde{M}}
\newcommand{\elln}{{\ell^2_n}}
\newcommand{\ellnn}{{\ell^2_{n+1}}}
\newcommand{\ella}{\ell^2_a}
\newcommand{\ellan}{\ell^2_{a, n}}
\newcommand{\ellann}{\ell^2_{a, n+1}}
\newcommand{\Phiin}{\Phi^i_n}

\newcommand{\norm}[1]{\left\lVert#1\right\rVert}
\newcommand{\normell}[1]{\left\lVert#1\right\rVert_{\ell^2}}

\newcommand{\normellnn}[1]{\left\lVert#1\right\rVert_{\ellnn}}

\title{On spectra and spectral measures of Schreier and Cayley graphs}
\author{Rostislav Grigorchuk, Tatiana Nagnibeda, Aitor Pérez}
\date{}
\begin{document}

\begin{abstract}
We are interested in various aspects of spectral rigidity of Cayley and Schreier graphs of finitely generated groups. For each pair of integers $d\geq 2$ and $m \ge 1$, we consider an uncountable family of groups of automorphisms of the rooted $d$-regular tree which provide examples of the following interesting phenomena. For $d=2$ and any $m\geq 2$, we get an uncountable family of non quasi-isometric Cayley graphs with the same Laplacian spectrum, absolutely continuous on the union of two intervals, that we compute explicitly. Some of the groups provide examples where the spectrum of the Cayley graph is connected for one generating set and has a gap for another.

For each $d\geq 3, m\geq 1$, we exhibit infinite Schreier graphs of these groups with the spectrum a Cantor set of Lebesgue measure zero union a countable set of isolated points accumulating on it. The Kesten spectral measures of the Laplacian on these Schreier graphs are discrete and concentrated on the isolated points. We construct moreover a complete system of eigenfunctions which are strongly localized. 
\end{abstract}

\maketitle

\section{Introduction}
\label{sec:introduction}

Cayley graphs and, more generally, Schreier graphs of finitely generated groups constitute an important class of examples in spectral graph theory. At the same time, the study of the Laplacian spectrum and spectral measures occupies a significant place in the theory of random walks on groups and more generally in geometric group theory. It is particularly interesting to understand how the spectra and spectral measures on Cayley and Schreier graphs depend on the algebraic structure and on the geometry of the group. There are also some natural rigidity questions, for instance, whether the spectrum determines the group in some way or, as formulated by Alain Valette \cite{Val}, ``Can one hear the shape of a group?". Another interesting question is how the spectrum depends on the chosen generating set or on the choice of weights on the generators. The spectral computations are notoriously difficult, and very few examples are known of infinite graphs, or of infinite families of finite graphs, where the spectrum has been explicitly computed. The qualitative results are also scarce. It follows from deep results in K-theory that the spectra are intervals for some classes of finitely generated torsion-free groups  \cite{HigKas}, and this is conjectured to be the case for all such groups. For groups that do contain a nontrivial element of finite order, the list of known shapes of spectra of Cayley graphs is very short: an interval, a union of an interval with one or two isolated points, a union of two disjoint intervals and one or two isolated points (free products of two finite cyclic groups \cite{CS86}), two disjoint intervals (Grigorchuk's group \cite{DG18}). A union of any finite number of disjoint intervals and one or two isolated points in the spectrum can appear in the case of anisotropic Laplacians on free products of several finite cyclic groups \cite{Kuh92}. Infinitely many gaps may appear in the spectrum of an anisotropic Laplacian on a lamplighter group~\cite{GS19}. The first examples of Schreier graphs whose spectrum is a Cantor set of Lebesgue measure zero or a union of such Cantor set and a countable set of isolated points accumulating on it were obtained in \cite{BG00}, and it is still open whether Cantor spectrum can occur on a Cayley graph. Even less is known about the spectral measure type. Lamplighter groups remain the only family of examples for which the spectral measure has been shown to be purely discrete (see~\cite{GriZ} for the original result on the lamplighter over $\Z$, and~\cite{LNW08} for a generalization to lamplighters with arbitrary bases).
Anisotropic Laplacians on lamplighters may have nontrivial singular continuous part~\cite{GV15}. An example of a Schreier graph of a self-similar group (the Hanoi towers group) with a nontrivial singular continuous part in the spectral measure appeared in~\cite{Qui07}. Examples of Schreier graphs with purely singular continuous spectra for anisotropic Laplacians were provided in \cite{GLN16, GLNS19}.

In this paper we study the spectra of Laplacians associated with certain self-similar group actions on rooted trees. Let $X = \{0, \dots, d-1\}$. The set of vertices of the tree $T_d$ is naturally identified with the set $X^*$ of finite words on $X$. Similarly, the boundary of the $d$-regular rooted tree $\partial T_d$ is in bijection with the set $X^\N$ of infinite sequences of elements of $X$. We shall alternatively use the notation $\partial T_d$, usually omitting the index $d$, or $X^\N$ to denote it. Given a finitely generated group $G$ acting by automorphisms on $T_d$, equipped with a natural finite set of generators $S$, we get a sequence of finite graphs $\{\Gamma_n\}_n$ (Schreier graphs of the action on finite levels of the tree) and a family $\{\Gamma_\xi\}_{\xi\in\partial T_d}$ of infinite Schreier graphs corresponding to the orbits of the action of $G$ by homeomorphisms on the boundary $\partial T_d$ of the tree. We will consider the normalized adjacency, or Markov, operator on the Cayley graph $\displaystyle M = \frac{1}{|S|}\sum_{s\in S} s :\ell^2(G)\rightarrow\ell^2(G)$, as well as its projections on the finite and infinite Schreier graphs: $M_n:\ell^2(G/H_n)\rightarrow\ell^2(G/H_n)$, where $H_n$ is the stabilizer subgroup of a vertex on the $n$-th level of the tree, and  $M_\xi: \ell^2(G/H_\xi)\rightarrow\ell^2(G/H_\xi)$, with $H_\xi$ the stabilizer of a point $\xi$ in the boundary $\partial T_d$. When a generating set is fixed we will often write $\spec(G)$ for the spectrum of the operator M on the Cayley graph.

The groups that we consider, the so-called spinal groups with the cyclic action at the root, are organized in uncountable families $\{G_\omega\}_{\omega\in\Omega_{d,m}}$. Here $d\geq 2$ denotes the degree of the regular rooted tree on which the group acts, $m\geq 1$ is an integer, and the groups in the family corresponding to a given pair $d,m$ are indexed by sequences in the alphabet consisting of all epimorphisms $(\Z/\Z_d)^m\rightarrow \Z/\Z_d$. They come with a natural choice of a set of generators that we call spinal generating set. See Section~\ref{sec:preliminaries} for the definition of spinal groups. In the case of $d=2$ the Schreier graphs  $\Gamma_\xi$ of spinal groups are infinite lines with some multiple edges and loops,  which makes the spectral analysis easier in this case. This is not the case anymore for $d\geq 3$,  and it turns out that the spectral properties of the operators $M_\xi$ are very different for  $d\geq 3$ as compared to $d=2$. All spinal groups with $d=2$ are of intermediate growth. For $d \ge 3$, this is known for some but not all of them (\cite{bart_pochon, bart_sunik_spinal, Fra19}). All spinal groups are amenable~\cite{JNS}. Hence, all their Schreier graphs are also amenable, and consequently, the spectrum of the operator $M_\xi$ does not depend of $\xi$~\cite{BG00}.

Below we prove that, in the case of $d=2$, the spectrum of the infinite Schreier graph $\Gamma_\xi$, denoted by $\spec(M_\xi)$, is a union of two intervals. The spectral measure, that we compute explicitly, is absolutely continuous with respect to the Lebesgue measure.  More interestingly, it  happens that for $d=2$ the spectrum $\spec(M_\xi)$ of $\Gamma_\xi$ coincides with the spectrum $\spec(M)$ of the Cayley graph of the group. We show that it is independent on $\omega\in \Omega_{d,m}$, and hence we obtain a negative answer to the question ``Can one hear the shape of a group?" by providing uncountable families of isospectral groups (see also \cite{DG18}). The groups $G_\omega$ with $d=2$ and $m=2$ all have nonequivalent growth functions~\cite{Gri84}, so there are uncountably many non quasi-isometric isospectral groups.


For spinal groups acting on the binary tree, we also investigate the dependence of the spectrum on the generating set. As mentioned above, the spectrum of both Schreier and Cayley graphs with respect to the spinal generating set is a union of two intervals. In the same time, there always exists a minimal generating set with the spectrum of the corresponding Schreier graph a Cantor set. For a certain subfamily of spinal groups acting on the binary tree that we denote $\{G_m\}$ (one group for each $m \ge 2$) there is also a minimal generating set with the spectrum of the corresponding Schreier graph an interval. For $m = 2, 3$ these examples provide groups with spectrum of the Cayley graphs an interval for one generating set and a union of two disjoint intervals for another generating set.

For $d\geq 3$ we also compute the spectrum of $\Gamma_\xi$. It is is a Cantor set of Lebesgue measure zero plus a countable set of points accumulating on it. The computations are inspired by the work of Bartholdi and Grigorchuk for one of these groups in \cite{BG00}. We also extend their computation of the empirical spectral measure, or density of states. Moreover, we go further and study the spectral measures for the operators $M_\xi$. Note that while the spectrum of $M_\xi$ doesn't depend on $\xi$, the spectral measures a priori do. We prove that for all $\xi$ in a certain (explicitly given) subset of full measure in $\partial T_d$, all the spectral measures of $M_\xi$ are discrete and concentrated on the set of isolated points in the spectrum. Moreover, we provide a complete system of eigenfunctions of $M_\xi$ and show that they are all of finite support.

Our main results are the following.

\begin{restatable}{thm}{spectrumcayleydtwo}
	\label{thm:spectrum_cayley_d2}
	Let $G$ be a spinal group with $d=2$ and $m\ge 2$. Then, for any $\xi \in X^\N$,
	\begin{equation}
	\spec(G) = \spec(M_\xi) = \left[- \frac{1}{2^{m-1}}, 0\right] \cup \left[1 - \frac{1}{2^{m-1}}, 1\right].
	\end{equation}
	
	Notice that, as $m \to \infty$, these spectra shrink from two intervals to two points.
\end{restatable}
\begin{restatable}{cor}{regularspectrumdtwo}
	\label{cor:regular_spectrum_d2}
	(see also~\cite{GD17}).
	There are uncountably many pairwise non quasi-isometric isospectral groups.
\end{restatable}

In the same time, we are able to find a different, minimal generating set, with the spectrum of the Schreier graph a Cantor set.

\begin{restatable}{cor}{generatingsetcantor}
	\label{cor:generating-set-cantor}
	
	For every spinal group $G_\omega$ with $d = 2$, $m \ge 2$ and $\omega \in \Omega_{d, m}$ there exists a minimal generating set $T \subset S$ for which $\spec(M_\xi^T)$ is a Cantor set of Lebesgue measure zero.
\end{restatable}

For two specific examples further analysis shows that the spectrum on the Cayley graph may have a gap or be connected, depending on the generating set. See Section 7 for the definition of \v{S}uni\'c's family of self-similar spinal groups and of the subfamily $\{G_m\}_{m \ge 2}$ acting on the binary tree.

\begin{restatable}{cor}{spectrumgego}
	\label{cor:spectrum-ge-go}
	For the Grigorchuk-Erschler group $G_2$ and Grigorchuk's overgroup $G_3$ the spectrum of the Cayley graph is a union of two disjoint intervals with respect to the spinal generating set and the interval $[-1, 1]$ with respect to the minimal \v{S}uni\'c generating set.
\end{restatable}

For spinal groups acting on trees of higher degree, we do not know the spectrum of the group, but we do know the spectrum $\spec(M_\xi)$ of the Schreier graphs. Consider the map $F(x) = x^2 - d(d-1)$, and denote $\psi(t) = \frac{1}{d^{m-1}}(|S|^2t^2 - |S|(|S| - 2)t - (|S| + d - 2))$.

\begin{restatable}{thm}{spectrumschreier}
	\label{thm:spectrum_schreier}
	Let $G$ be a spinal group with $d \ge 2$ and $m \ge 1$, generated by the spinal generators. Then, for any $\xi \in X^\N$,
	\begin{equation}
	\spec(M_\xi) = \left\lbrace \frac{|S|-d}{|S|}\right\rbrace \cup \psi^{-1}\left(\overline{\bigcup_{n \ge 0}F^{-n}(0)}\right).
	\end{equation}
	
	For $d=2$, we have $\spec(M_\xi) = \left[- \frac{1}{2^{m-1}}, 0\right] \cup \left[1 - \frac{1}{2^{m-1}}, 1\right]$.
	
	For $d>2$, we can decompose
	\[
	\spec(M_\xi) = \spec^0(M_\xi) \cup \spec^\infty(M_\xi),
	\]
	with $\spec^\infty(M_\xi)$ being a Cantor set and $\displaystyle \spec^0(M_\xi) = \left\lbrace \frac{|S|-d}{|S|}\right\rbrace\cup\psi^{-1}\left(\bigcup_{n \ge 0}F^{-n}(0)\right)$ being a countable set of isolated points accumulating on this Cantor set.
\end{restatable}

Notice that this spectrum is the preimage by the quadratic map $\psi$ of the set $\overline{\bigcup_{n \ge 0}F^{-n}(0)}$ of preimages of $0$ under $F$ and of its closure, the Julia set of $F$ (plus an isolated point). For $d=2$, the Julia set of $F(x) = x^2-2$ is the interval $[-2, 2]$, which contains $\cup_{n \ge 0} F^{-n}(0)$, hence its preimage by $\psi$ is the union of two intervals. For $d>2$, however, the Julia set of $F$ is a Cantor set, and is disjoint with $\cup_{n \ge 0} F^{-n}(0)$. Therefore, its preimage by $\psi$ is again a Cantor set, and $\psi^{-1}(\cup_{n \ge 0} F^{-n}(0))$ is a countable set of points accumulating on this Cantor set.

Our proof of Theorem~\ref{thm:spectrum_schreier} follows the strategy developed in~\cite{BG00} for some examples and generalizes their technique. First, we use the Schur complement method to find a recurrence between the spectrum at level $n$ and the spectrum at level $n-1$. Then we solve this recurrence to completely describe these finite spectra in Theorem~\ref{thm:finite_spec}. Finally, we use this result to compute $\spec(M_\xi)$ in the proof of Theorem~\ref{thm:spectrum_schreier}, using the fact that the Schreier graphs on the boundary are limits of those on finite levels.

The next step after identifying the spectrum is to study the spectral measures. Our results involve both the classical spectral measures on the infinite graphs and the empirical spectral measure (density of states), obtained by the finite approximations of the infinite graph.

Let $\CH$ be a Hilbert space, let $T: \CH \to \CH$ be a self-adjoint linear operator, and let $f \in \CH$. Then there is a unique positive measure $\mu_f$ on $\spec(T)$ such that for every $n \ge 0$
\[
\int_{\spec(T)}x^n d\mu_f(x) = \langle T^n f, f \rangle,
\]
called the \emph{spectral measure} of the operator $T$ associated with $f$. Notice that $\mu_f(\spec(T)) = \norm{f}^2$, so $\mu_f$ is a finite measure. Let $\Gamma$ be a graph. For every vertex $p$ of $\Gamma$, we call the spectral measure of the Markov operator for the simple random walk on $\Gamma$ $\mu_p = \mu_{\delta_p}$ its \emph{Kesten spectral measure}, where $\delta_p$ is $1$ at $p$ and vanishes everywhere else (see e.g. I.1.C in~\cite{W00}). The $n$-th moment of the Kesten spectral measure $\mu_p$ is the probability for the random walk to return to $p$ in $n$ steps. Such measures for Markov operators were first considered by Kesten~\cite{K59}. In this paper we will consider the measures $\mu_\eta$ for the vertices $\eta$ of Schreier graphs $\Gamma_\xi$ of points $\xi \in X^\N$.

These graphs are obtained as limits of the finite Schreier graphs $\Gamma_n$ on the levels of the tree, hence we also study the empirical spectral measure, or density of states. Given a sequence of finite graphs $\{Y_n\}_n$, for every $n \ge 0$ let $\nu_n$ be the counting measure on the spectrum of the Markov operator $M_n$ on $Y_n$:
\[
\nu_n = \frac{1}{|Y_n|}\sum_{\lambda \in \spec(M_n)}\delta_\lambda,
\]
where $\delta_x$ is the Dirac measure at $x$, and the eigenvalues are counted with multiplicity. Following~\cite{GZ04}, we call the weak limit $\nu$ of the measures $\nu_n$ the \emph{empirical spectral measure} or \emph{density of states} of $\{Y_n\}_n$.

\begin{restatable}{thm}{KNSspectralmeasure}
	\label{thm:KNS_spectral_measure}
	
	Let $\nu$ be the density of states of $\{\Gamma_n\}_n$. If $d = 2$, $\nu$ is absolutely continuous with respect to the Lebesgue measure. Its density is given by the function
	\begin{equation}
		g(x) = \frac{|2^{m-1} - 1 - 2^mx|}{\pi\sqrt{x(1-x)(2^mx + 2)(2^mx + 2 - 2^m)}}.
		\label{eq:dens_states}
	\end{equation}

	If $d \ge 3$, then $\nu$ is discrete. More precisely,
	\[
	\nu = \frac{d-2}{d}\delta_{\frac{|S|-d}{|S|}} + \sum_{n \ge 0}\frac{d-2}{d^{n+2}}\sum_{x \in \psi^{-1}(F^{-n}(0))}\delta_x.
	\]
\end{restatable}

Kesten spectral measures also have different type for the cases $d=2$ and $d \ge 3$. In the binary case, linearity of the Schreier graphs and the fact that all $\Gamma_\xi$ are isomorphic (as unlabeled graphs), for $\xi \in X^\N$ not in the orbit of $1^\N$, allow us to conclude that all the Kesten spectral measures $\mu_\xi$ are equal with the exception of the orbit of $1^\N$. For $d\geq 3$, we prove that $M_\xi$ possesses a complete system of eigenfunctions of finite support, corresponding to the isolated eigenvalues $spec^0$, for a set of points $\xi$ in $X^\N$ of full measure.

\begin{restatable}{prop}{Kestenspectralmeasurebinary}
	\label{prop:Kesten_spectral_measure_binary}
	Let $d=2$ and let $\xi \in X^\N$. The spectral measure $\mu_\xi$ is absolutely continuous with respect to the Lebesgue measure. For every $\xi$ not coorbital to $1^\N$, $\mu_\xi$ coincides with the density of states $\nu$ (see equation~(\ref{eq:dens_states})). In addition, $\mu_{1^\N}$ has density
	\begin{equation}
		h(x) = \frac{|x(2^mx + 2)|}{\pi\sqrt{x(1-x)(2^mx + 2)(2^mx + 2 - 2^m)}}.
		\label{eq:dens_1N}
	\end{equation}
\end{restatable}

For $d \ge 3$, a study of the eigenfunctions on both finite and infinite Schreier graphs yields that the operator $M_\xi$ has discrete Kesten spectral measures and the eigenfunctions are strongly localized.

Let $\sigma: X^\N \to X^\N$ be the map which removes the first letter of any point in $X^\N$. For any given $\xi \in X^\N$, we define $I_\xi = \{ n \in \N \mid \forall r \ge 0, \: (d-1)^r0 \textrm{ is not a prefix of } \sigma^n(\xi) \}$, and consider the subset $W$ of $X^\N$ defined as $W = \{\xi \in X^\N \mid k, k+1 \in I_\xi \textrm{ for infinitely many } k \}$.

\begin{restatable}{thm}{Kestenspectralmeasure}
	\label{thm:Kesten_spectral_measure}
	For every $d \ge 3$, the set $W$ has uniform Bernoulli measure $1$ in $X^\N$. For every spinal group defined by $d \ge 3$, $m \ge 1$ and $\omega \in \Omega_{d, m}$, and every $\xi \in W$, the operator $M_\xi$ has pure point spectrum. More precisely, it possesses a complete system of finitely supported eigenfunctions corresponding to the eigenvalues that form $\spec^0(M_\xi)$.
\end{restatable}
Note that all graphs for which Theorem~\ref{thm:Kesten_spectral_measure} applies are one-ended, and that the subset $W \subset X^\N$ does not depend on $m$ or $\omega$, but only on $d$.

Let us note that in the case of spinal groups with $d \ge 3$ and $m = 1$ it is also possible to compute the spectrum directly via renormalization of the infinite graph, without finite approximation. This method was used on some self-similar graphs by Malozemov and Teplyaev in~\cite{MT03} and on an example of a Schreier graph of the Hanoi towers group by Quint in~\cite{Qui07}. One advantage of this method is that it allows simultaneously to compute the spectrum of all points in the space of Schreier graphs, i.e., not only of the graphs $\{\Gamma_\xi\}_{\xi \in \partial T}$, but also of their accumulation points in the space of labeled rooted graphs (see~\cite{BDN16}). These additional graphs are of special interest, as the spectral measures on them have a nontrivial singular continuous component. This approach will be detailed in a subsequent paper.

\vspace{1em}

The structure of the paper is as follows. In Section~\ref{sec:preliminaries}, we give some basic definitions about spinal groups, Schreier graphs and Markov operators, as well as some examples. Sections~\ref{sec:finite} and~\ref{sec:infinite} are devoted to proving Theorem~\ref{thm:spectrum_schreier}. We conclude the section by computing the density of states in Theorem~\ref{thm:KNS_spectral_measure}. In Section~\ref{sec:pure_point_eigenfunctions}, we discuss the spectral measures on the Schreier graphs and the eigenfunctions of the Markov operators. In particular, we prove the equality of the spectral measures with the density of states if $d=2$, and for $d\ge 3$ we show that the Kesten spectral measures are discrete and concentrated on the set of isolated eigenvalues, for any $\xi$ in a certain explicitly described measure one set of boundary points. We do that by explicitly finding the eigenfunctions of $M_\xi$ and showing that they form a complete set. Moreover, all of them are finitely supported.
In Section~\ref{sec:spectra_cayley}, we show the equality between $\spec(G)$ and $\spec(M_\xi)$ for spinal groups acting on the binary tree. 
The fact that the spectra on $\Gamma_\xi$ do not depend on $\omega\in\Omega_{d,m}$ ensures that we obtain an uncountable family of groups with the same spectrum.
Finally, in Section~\ref{sec:dependence_gen_set} we study the dependence of spectra on generating sets, give some examples and prove Corollary~\ref{cor:generating-set-cantor} and Corollary~\ref{cor:spectrum-ge-go}.

\section{Preliminaries}
\label{sec:preliminaries}

\begin{dfn}(\textbf{Spinal groups})
	Let $d \ge 2$ and $X = \{0, 1, \dots, d-1\}$. We denote by $T$ the $d$-regular rooted tree, whose vertices are in bijection with $X^*$ (the set of all finite words on the alphabet $X$), and by $\partial T$ its boundary, in bijection with $X^\N$ (the set of all right-infinite words on the alphabet $X$).
	
	Choose an integer $m \ge 1$, and let $A = \langle a \rangle = \Z/d\Z$ and $B = (\Z/d\Z)^m$. Denote by $\Epi(B, A)$ the set of epimorphisms from $B$ to $A$, and define $\Omega = \Omega_{d,m} \subset \Epi(B, A)^\N$ to be the set of sequences of epimorphisms satisfying the condition
	\begin{equation}
	\label{eq:kernel-condition}
	\forall i \ge 0, \quad \bigcap\limits_{j\ge i} \Ker(\omega_j) = 1.
	\end{equation}

	Following~\cite{bart_sunik_spinal} and \cite{bart_grig_sunik_branch}, we define for every $\omega = \omega_0\omega_1\dots \in \Omega_{d, m}$ the spinal group $G_\omega$ as the subgroup of $\Aut(T)$ generated by $A$ and $B$. Here, by abuse of notation, $A$ and $B$ denote the subgroups of the following automorphisms:
	\[
	a(v_0 v_1 \dots) = (v_0 + 1 \text{ mod } d) v_1 \dots
	\]\[
	b(v_0 v_1 \dots) = \left\{ \begin{array}{ll}
	v_0v_1 \dots v_n \: \omega_n(b)(v_{n+1}) \: v_{n+2}\dots & \text{if } v_0\dots v_n = (d-1)^n0 \\ v_0v_1\dots & \text{otherwise} \end{array} \right..
	\]
	The condition~(\ref{eq:kernel-condition}) implies that the action of $G_\omega$ on $T$ is faithful. The automorphism $a$ permutes the subtrees under the root cyclically. Elements in $B$ fix all vertices in the \emph{spine}, the rightmost ray of the tree, whose vertices are words of the form $(d-1)^n$. Moreover, their action is trivial everywhere except in the subtrees under vertices of the form $(d-1)^n0$. Notice that the action of $G_\omega$ on the tree is transitive in every level, and orbits in $X^\N$ are cofinality classes.
\end{dfn}

If not stated otherwise, we will consider $G_\omega$ with the \emph{spinal generating set} $S = (A\cup B) \setminus \{1\}$. Recall that $|S| = d^m + d - 2$.

\begin{dfn}(\textbf{Schreier graphs})
Let $G$ be a group finitely generated by $S$, and $H \le G$. The \emph{Schreier graph} associated with $H$, denoted $\Sch(G, H, S)$, is the graph whose vertices are lateral classes $G/H$ and, for every $s \in S$ and $gH \in G/H$, there is an edge from $gH$ to $sgH$ labeled by $s$.

In our case, we will always choose $H$ to be the stabilizer in $G_\omega$ of some vertex $u$ of $X^*$ or $X^\N$.	Since the action is transitive at every level, if $u \in X^n$, the graph will not depend on the choice of $u$, so we will write $\Gamma_n = \Sch(G_\omega, Stab_{G_\omega}(u), S)$. Finally, for $\xi \in X^\N$, we will write $\Gamma_\xi = \Sch(G_\omega, Stab_{G_\omega}(\xi), S)$.

\end{dfn}

\begin{dfn}(\textbf{Markov operator})
For a graph $\Gamma$, the Markov operator is the normalized adjacency operator $M: \ell^2(V(\Gamma)) \to \ell^2(V(\Gamma))$, given by $Mf(x) = \frac{1}{\deg(x)} \sum_{y \sim x} f(y)$.

We will be particularly interested in Markov operators on Schreier or Cayley graphs of spinal groups. For $n \ge 0$, $M_n: \ell^2(X^n) \to \ell^2(X^n)$ will denote the Markov operator on $\Gamma_n$, defined by $M_nf(v) = \frac{1}{|S|}\sum_{s \in S} f(sv)$, and, for $\xi \in X^\N$, $M_\xi: \ell^2(G\xi) \to \ell^2(G\xi)$ will denote the Markov operator on $\Gamma_\xi$, defined by $M_\xi f(\eta) = \frac{1}{|S|}\sum_{s \in S} f(s \eta)$. Finally, $Mf(g) = \frac{1}{|S|}\sum_{s \in S} f(sg)$ acting on $\ell^2(G)$ is the Markov operator on the Cayley graph of $G$ with respect to $S$. We will denote its spectrum by $\spec(G)$ when there is no confusion as to the choice of $S$.

\end{dfn}

The family of spinal groups contains some well-known examples of groups of intermediate growth. In particular the first Grigorchuk group is obtained by taking $d = 2$, $m = 2$ and $\omega = (\pi_d\pi_c\pi_b)^\N$ where $A = \{1, a\}$, $B = \{1, b, c, d\}$ and $\pi_x: B \to A$ is the epimorphism mapping $x$ to $1$. The Fabrykowski-Gupta group corresponds to the case $d=3$, $m=1$ and $\omega = \pi^\N$, where $A = \{1, a, a^2\}$, $B = \{1, b, b^2\}$ and $\pi: B \to A$ is the epimorphism mapping $b$ to $a$.

\section{Spectra of finite Schreier graphs}
\label{sec:finite}

Let $G = G_\omega$ be a spinal group with parameters $d \ge 2$, $m \ge 1$ and $\omega \in \Omega_{d, m}$, as defined above.

\input{img/grigorchuk-3}
\input{img/gupta-fabrykowski-3}

Our goal in this section is to compute $\spec(M_n)$, the spectrum of the Schreier graph associated to the action of $G$ on the $n$-th level of the tree, which is a finite graph on $d^n$ vertices. Typical examples of finite Schreier graphs of spinal groups with $d = 2$ and $d \ge 3$ can be found in Figures~\ref{fig:grigorchuk-3} and~\ref{fig:gupta-fabrykowski-3}, respectively.

In order to simplify the computations, we will compute the spectrum of $\M_n = |S|M_n$ of the adjacency matrix of $\Gamma_n$, and then normalize it dividing by $|S|$. We will first find a recurrence relation between $\spec(\M_n)$ and $\spec(\M_{n-1})$, then we will solve this recurrence to explicitly get $\spec(\M_n)$ and finally in Section~\ref{sec:infinite} we will take the limit to find $\spec(M_\xi)$.

Since the vertices of $\Gamma_n$ are words on $X = \{0, \dots, d-1\}$ of length $n$, we will order them lexicographically. Let us first start by computing the matrix of $\M_n$.  We write it as a $d \times d$ block matrix where each block is a matrix of size $d^{n-1} \times d^{n-1}$. A block denoted by a scalar is the corresponding multiple of the identity matrix $I_{d^{n-1}}$.

\begin{lem}
	\label{lem:adjacency_matrix}
	
	Let $A_0 = d-1$ and $B_0 = d^m - 1$. Define, for $n \ge 1$, the following matrices in $M_{d^n, d^n}(\R)$.
	\[
	A_n = \left(
	\begin{array}{ccccc}
	0 & 1 & \dots & 1 & 1 \\
	1 & 0 & \dots & 1 & 1 \\
	\vdots & \vdots & \ddots &  \vdots &  \vdots \\
	1 & 1 & \dots  & 0 & 1 \\
	1 & 1 &  \dots & 1 & 0
	\end{array}
	\right),\]\[
	B_n = \left(
	\begin{array}{ccccc}
	d^{m-1}A_{n-1} + d^{m-1} - 1 &  &  &  &  \\
	& d^{m-1} - 1 &  &   &  \\
	&   & \ddots  &   &   \\
	&  &   & d^{m-1} - 1 &  \\
	&  &   &  & B_{n-1}
	\end{array}
	\right).
	\]
	
	Then, the matrix of $\M_n$ is $A_n + B_n$.
\end{lem}
\begin{proof}
	In order to write the adjacency matrix of $\Gamma_n$, we will first write the adjacency matrices of each of the generators we consider. For a generator $s \in S$, its adjacency matrix for $\Gamma_n$ is denoted $s_n$, and its coefficient $(u, v)$ is $1$ if $s(u) = v$ and $0$ otherwise.
	
	Let $a$ be the generator of $A$ permuting the subtrees of the first level cyclically. We can write the adjacency matrix of this generator by blocks as
	\[
	a_0 = 1, \qquad
	a_n = 
	\left(
	\begin{array}{ccccc}
	0 & 1 & &   &  \\
	& 0 & 1 &   &  \\
	& &  \ddots & \ddots &   \\
	&  & & 0 & 1 \\
	1 &  &   & & 0
	\end{array}
	\right), \quad \forall n \ge 1.
	\]
	
	Now, for any $b\in B$ and $k \ge 0$, if we write the matrix
	\[
	b_{0,k} = 1 \qquad 
	b_{n,k} = 
	\left(
	\begin{array}{ccccc}
	\omega_k(b)_{n-1} &  &  &  &  \\
	& 1 &  &   &  \\
	&   & \ddots  &   &   \\
	&  &   & 1 &  \\
	&  &   &  & b_{n-1, k+1}
	\end{array}
	\right), \quad \forall k\ge0, 
	\]
	then the adjacency matrix of $b$ is $b_n = b_{n, 0}$.
	
	These matrices have size $d^n \times d^n$, and every block in the matrices above is a matrix of size $d^{n-1} \times d^{n-1}$. In order to simplify notation, from now on we will omit the identity matrix and write just the scalar multiplying it, as the dimensions should be clear from the context.
	
	Now notice that, for every $n \ge 0$, $\sum\limits_{i=1}^{d-1} a_n^i = A_n$. Similarly, we have $\sum\limits_{b \in B\setminus\{1\}}b_n = B_n$. Indeed,
	\[
	\sum\limits_{b \in B\setminus\{1\}}b_{n,k} = 
	\left(
	\begin{array}{ccccc}
	\sum\limits_{b \in B\setminus\{1\}}\omega_k(b)_{n-1} &  &  &  &  \\
	& d^m - 1 &  &   &  \\
	&   & \ddots  &   &   \\
	&  &   & d^m - 1 &  \\
	&  &   &  & \sum\limits_{b \in B\setminus\{1\}}b_{n-1,k+1}
	\end{array}
	\right) = \]
	\[ = \left(
	\begin{array}{ccccc}
	d^{m-1}A_{n-1} + d^{m-1} - 1 &  &  &  &  \\
	& d^m - 1 &  &   &  \\
	&   & \ddots  &   &   \\
	&  &   & d^m - 1 &  \\
	&  &   &  & \sum\limits_{b \in B\setminus\{1\}}b_{n-1,k+1}
	\end{array}
	\right).
	\]
	The sum in the first block does not depend on $k$, since $\omega_k$ is an epimorphism and all elements of $A$ have exactly $d^{m-1}$ preimages. Hence we can inductively conclude that $\sum\limits_{b \in B\setminus\{1\}}b_n = B_n$.
	
	Finally, the adjacency matrix of $\Gamma_n$ is $\sum\limits_{s \in S}s_n = \sum\limits_{i=1}^{d-1} a_n^i + \sum\limits_{b \in B\setminus\{1\}}b_n = A_n + B_n$.
\end{proof}

If we now try to find the characteristic polynomial of $\M_n$, we will not find any explicit relation with that of $\M_{n-1}$. Instead, we consider the matrix
\[
Q_n(\lambda, \mu) := B_n + \lambda A_n - \mu.
\]
These additional parameters will allow us to find a relation between the determinant of $Q_n(\lambda, \mu)$ and $Q_{n-1}(\lambda', \mu')$, for some different $\lambda'$ and $\mu'$. According to Lemma~\ref{lem:adjacency_matrix}, by setting $\lambda = 1, \mu = 0$, we recover the matrix of $\M_n$, so more specifically we want to find $\spec(\M_n) = \left\{ \mu \mid |Q_n(1, \mu)| = 0 \right\}$.

As mentioned above, the strategy consists of two steps. First, we will prove a relation between the determinants of $Q_n$ and $Q_{n-1}$ (Proposition~\ref{prop:recurrence}). Second, we will solve this recurrence to find a factorization of the determinant of $Q_n$ (Proposition~\ref{prop:factorization}).

Before, as our computations will involve matrices of the form $rA_n + s$, let us start with the following Lemma, which will be useful later on.

\begin{lem}
	\label{lem:rAns}
	Let $r, s, r', s' \in \R$. Then,
	
	\begin{enumerate}
		\item $A_n^2 = (d-2)A_n + d - 1$.
		\item $|rA_n + s| = \left[(s - r)^{d-1}(s + (d-1)r)\right]^{d^{n-1}}$.
		\item $(rA_n + s)^{-1} = \frac{rA_n - (d-2)r - s}{(r-s)(s + (d-1)r)}$.
		\item $(rA_n + s)(r'A_n + s') = [(d-2)rr' + rs' + r's]A_n + (d-1)rr' + ss'$.
	\end{enumerate}
\end{lem}
\begin{proof}
	For $(1)$, if we square $A_n$ then we will get a sum of $d-1$ ones for elements in the diagonal and $d-2$ ones for the rest, which shows the claim.
	
	For $(2)$, we have
	\[
	|rA_n + s| = \left| \begin{array}{ccccc}
	s & r & \dots & r & r \\
	r & s & \dots & r  & r \\
	\vdots & \vdots  & \ddots  &  \vdots & \vdots  \\
	r & r & \dots & s & r \\
	r & r & \dots & r & s
	\end{array} \right| = \left| \begin{array}{ccccc}
	s-r & 0 & \dots & 0 & r-s \\
	0 & s-r & \dots & 0  & r-s \\
	\vdots & \vdots  & \ddots  &  \vdots & \vdots  \\
	0 & 0 & \dots & s-r & r-s \\
	r & r & \dots & r & s
	\end{array} \right|	= \]\[ =
	(s-r)^{(d-1)d^{n-1}} \left|s - \frac{(d-1)r(r-s)}{s-r} \right| = \left[(s - r)^{d-1}(s + (d-1)r)\right]^{d^{n-1}}.
	\]
	
	For $(3)$, we can verify, using $(1)$, that
	\[
	(rA_n + s)\left(rA_n - (d-2)r - s\right) = r^2A_n^2 - (d-2)r^2A_n - (d-2)rs - s^2 = \]\[
	= r^2\left((d-2)A_n + d -1\right) - (d-2)r^2A_n - (d-2)rs - s^2 = (r-s)(s + (d-1)r).
	\]
	
	Claim $(4)$ can be checked directly, again using $(1)$.
\end{proof}

\begin{prop}
\label{prop:Q0-Q1}

For $n = 0$ and $n = 1$, we have
\[
|Q_0(\lambda, \mu)| = \alpha + \lambda \quad \text{and} \quad
|Q_1(\lambda, \mu)| = (\alpha + \lambda)\beta^{d-1},
\]
where
\[
\alpha = \alpha(\lambda, \mu) := d^m - 1 - \mu + (d-2)\lambda
\]
and
\[
\beta = \beta(\lambda, \mu) := d^m - 1 - \mu - \lambda.
\]
\end{prop}
\begin{proof}
By direct computation,
\[
\left|Q_0(\lambda, \mu)\right| = B_0 + \lambda A_0 - \mu = d^m - 1 - \mu + (d-1)\lambda = \alpha + \lambda,
\]
\[
\left|Q_1(\lambda, \mu)\right| = \left|B_1 + \lambda A_1 - \mu \right| = \left|\lambda A_1 + d^m - 1 - \mu \right| = \]\[ = (d^m - 1 - \mu + (d-1)\lambda)(d^m - 1 - \mu - \lambda)^{d-1} = (\alpha + \lambda)\beta^{d-1}.
\]
\end{proof}

We are now ready to compute the determinant of $Q_n(\lambda, \mu)$ for $n \ge 2$:

\begin{prop}
\label{prop:recurrence}

For $n \ge 2$, we have
\[
\left|Q_n(\lambda, \mu)\right| = (\alpha\beta^{d^2 - 3d + 1}\gamma^{d-1})^{d^{n-2}} \left|Q_{n-1}(\lambda', \mu')\right|,
\]
with $\alpha$ and $\beta$ as in Proposition~\ref{prop:Q0-Q1},
\[
\lambda' := \frac{d^{m-1}\beta}{\alpha\gamma}\lambda^2 \quad \text{and} \quad \mu' := \mu + \frac{(d-1)\delta}{\alpha\gamma}\lambda^2,
\]
where
\scriptsize
\[
\gamma = \gamma(\lambda, \mu) := \mu^2 - ((d-3)\lambda + d^m - 2)\mu - ((d-2)\lambda^2 + (d-3)\lambda + d^m - 1),
\]
\[
\delta = \delta(\lambda, \mu) := \mu^2 - ((d-3)\lambda + d^m + d^{m-1} - 2)\mu - ((d-2)\lambda^2 + (d^{m-1}+d-3)\lambda - d^{2m-1} + d^m + d^{m-1} - 1).
\]
\normalsize
\end{prop}
\begin{proof}
We start computing the determinant of $|Q_n(\lambda, \mu)|$ directly, performing elementary transformations of rows and columns in determinants.
\[
\left|Q_n(\lambda, \mu)\right| = \left|\begin{array}{ccccc}
d^{m-1}A_{n-1} + d^{m-1} - 1 - \mu & \lambda & \dots & \lambda & \lambda \\
\lambda & d^m - 1 - \mu & \dots & \lambda & \lambda \\
\vdots & \vdots & \ddots & \vdots & \vdots \\
\lambda & \lambda & \dots & d^m - 1 - \mu & \lambda \\
\lambda & \lambda & \dots & \lambda & B_{n-1} - \mu
\end{array}\right| =
\]\[ =
\left|\begin{array}{ccc|cc}
\beta + \lambda & \dots & \lambda & \lambda & \lambda \\
\vdots & \ddots & \vdots & \vdots & \vdots \\
\lambda & \dots & \beta + \lambda & \lambda & \lambda \\ \hline
\lambda & \dots & \lambda & d^{m-1}A_{n-1} + d^{m-1} - 1 - \mu & \lambda \\
\lambda & \dots & \lambda & \lambda & B_{n-1} - \mu
\end{array}\right| =
\]\[
= \left|\begin{array}{ccc|cc}
\beta &  &  & \lambda + \mu + 1 - d^{m-1} - d^{m-1}A_{n-1} & 0 \\
& \ddots &  & \vdots & \vdots \\
&  & \beta & \lambda + \mu + 1 - d^{m-1} - d^{m-1}A_{n-1} & 0 \\ \hline
\lambda & \dots & \lambda & d^{m-1}A_{n-1} + d^{m-1} - 1 - \mu & \lambda \\
\lambda & \dots & \lambda & \lambda & B_{n-1} - \mu
\end{array}\right| =
\]\[
= \left|\begin{array}{ccc|cc}
\beta &  &  & \lambda + \mu + 1 - d^{m-1} - d^{m-1}A_{n-1} & 0 \\
& \ddots &  & \vdots & \vdots \\
&  & \beta & \lambda + \mu + 1 - d^{m-1} - d^{m-1}A_{n-1} & 0 \\ \hline
0 & \dots & 0 & d^{m-1}A_{n-1} + d^{m-1} - 1 - \mu - \frac{(d-2)\lambda(\lambda + \mu + 1 - d^{m-1} - d^{m-1}A_{n-1})}{\beta} & \lambda \\
0 & \dots & 0 & \lambda - \frac{(d-2)\lambda(\lambda + \mu + 1 - d^{m-1} - d^{m-1}A_{n-1})}{\beta} & B_{n-1} - \mu
\end{array}\right| =
\]\[
= (\beta^{d-2})^{d^{n-1}}\left|\begin{array}{cc}
d^{m-1}A_{n-1} + d^{m-1} - 1 - \mu - \frac{(d-2)\lambda(\lambda + \mu + 1 - d^{m-1} - d^{m-1}A_{n-1})}{\beta} & \lambda \\
\lambda\left(1 - \frac{(d-2)(\lambda + \mu + 1 - d^{m-1} - d^{m-1}A_{n-1})}{\beta}\right) & B_{n-1} - \mu
\end{array}\right| =
\]\scriptsize\[
= (\beta^{d-3})^{d^{n-1}}\left|\begin{array}{cc}
\beta(d^{m-1}A_{n-1} + d^{m-1} - 1 - \mu) - (d-2)\lambda(\lambda + \mu + 1 - d^{m-1} - d^{m-1}A_{n-1}) & \lambda^2 \\
\beta - (d-2)(\lambda + \mu + 1 - d^{m-1} - d^{m-1}A_{n-1}) & B_{n-1} - \mu
\end{array}\right| =
\]\normalsize\[
= (\beta^{d-3})^{d^{n-1}}\left|\begin{array}{cc}
d^{m-1}(\alpha - \lambda)(A_{n-1} + 1) + \gamma & \lambda^2 \\
(d-2)d^{m-1}(A_{n-1} - (d-1)) + (d-1)\beta & B_{n-1} - \mu
\end{array}\right|.
\]

We set for convenience
\[
C_n := d^{m-1}(\alpha - \lambda)(A_n + 1) + \gamma,
\]\[
D_n := (d-2)d^{m-1}(A_n - (d-1)) + (d-1)\beta.
\]
We continue the computation of the determinant by taking the first Schur complement. Namely, whenever $P$ is invertible, we have the equality $\left|\begin{array}{cc} P & Q \\ R & S \end{array}\right| = |P||S - RP^{-1}Q|$: 
\[
\left|Q_n(\lambda, \mu)\right| = (\beta^{d-3})^{d^{n-1}}\left|\begin{array}{cc}
C_{n-1}  & \lambda^2 \\
D_{n-1} & B_{n-1} - \mu
\end{array}\right| = 
\]\[ = (\beta^{d-3})^{d^{n-1}}\left|C_{n-1}\right|\left|B_{n-1} - \mu - \lambda^2 D_{n-1}C_{n-1}^{-1}\right|.
\]

Let us now compute these two determinants. First, using Lemma~\ref{lem:rAns} with $r = d^{m-1}(\alpha - \lambda)$ and $s = \gamma + d^{m-1}(\alpha - \lambda)$, we obtain
\[
\left| C_n \right| = \left[(s-r)^{d-1}(s + (d-1)r)\right]^{d^{n-1}} = (\alpha \beta \gamma^{d-1})^{d^{n-1}},
\]
as well as
\[
C_n^{-1} = \frac{-1}{\alpha\beta\gamma}\left[d^{m-1}(\alpha - \lambda)(A_n - (d-1)) - \gamma \right].
\]
Similarly, again by Lemma~\ref{lem:rAns} but now with
\[
\begin{array}{ll}
r = (d-2)d^{m-1}, & s = (d-1)\left(\beta - (d-2)d^{m-1}\right), \\
r' = d^{m-1}(\alpha - \lambda), & s' = - (\gamma + (d-1)d^{m-1}(\alpha - \lambda)),
\end{array}
\]
we find
\[
D_nC_n^{-1} = \frac{-1}{\alpha\gamma}\left[d^{m-1}\beta A_n - (d-1)\delta \right].
\]
Indeed,
\[
(d-2)rr' + rs' + r's =
\]\[
= (d-2)rr' - r(\gamma + (d-1)r') + r'(d-1)(\beta - r) =
\]\[
= -drr' - r\gamma + (d-1)\beta r' =
\]\[
= d^{m-1}\left[(d-1)\beta(\alpha - \lambda) - (d-2)(\gamma + d^m(\alpha - \lambda)) \right] =
\]\[
= d^{m-1}\beta[\alpha - (d-1)\lambda] = d^{m-1}\beta^2,
\]
and
\[
(d-1)rr' + ss' =
\]\[
= (d-1)rr' - (d-1)(\beta - r)(\gamma + (d-1)r') =
\]\[
= (d-1)(drr' + r\gamma - \beta(\gamma + (d-1)r')) =
\]\[
= (d-1)(r(\gamma + dr') - \beta(\gamma + (d-1)r')) =
\]\[
= (d-1)(r\alpha\beta - \beta(\gamma + (d-1)r')) =
\]\[
= -(d-1)\beta(\gamma + (d-1)d^{m-1}(\alpha - \lambda) - (d-2)d^{m-1}\alpha) =
\]\[
= -(d-1)\beta(\gamma + d^{m-1}(\alpha + (d-1)\lambda)) =
\]\[
= -(d-1)\beta(\gamma + d^{m-1}\beta) = -(d-1)\beta\delta.
\]

Therefore,
\[
B_{n-1} - \mu - \lambda^2 D_{n-1}C_{n-1}^{-1} =
\]\[
B_{n-1} - \mu + \frac{\lambda^2}{\alpha\gamma}\left[d^{m-1}\beta A_{n-1} - (d-1)\delta \right] =
\]\[
B_{n-1} + \frac{d^{m-1}\beta}{\alpha\gamma}\lambda^2 A_{n-1} - \left(\mu + \frac{(d-1)\delta}{\alpha\gamma}\lambda^2\right) =
\]\[
= B_{n-1} + \lambda' A_{n-1} - \mu' =
\]\[
= Q_{n-1}(\lambda', \mu').
\]

Finally, we conclude the computation of the determinant of $Q_n(\lambda, \mu)$:
\[
\left|Q_n(\lambda, \mu)\right| = (\beta^{d-3})^{d^{n-1}}\left|C_{n-1}\right|\left|B_{n-1} - \mu - \lambda^2 D_{n-1}C_{n-1}^{-1}\right| =
\]\[
= (\beta^{d-3})^{d^{n-1}} (\alpha \beta \gamma^{d-1})^{d^{n-2}} \left| Q_{n-1}(\lambda', \mu') \right| =
\]\[
= (\alpha \beta^{d^2 - 3d + 1} \gamma^{d-1})^{d^{n-2}} \left| Q_{n-1}(\lambda', \mu') \right|.
\]

\end{proof}

This concludes the first part of the strategy, finding a recurrence relation between the determinants of $Q_n(\lambda, \mu)$ and $Q_{n-1}(\lambda, \mu)$ via the Schur complement. For the next part, we need to unfold this recurrence relation to get a factorization of $|Q_n(\lambda, \mu)|$. Proposition~\ref{prop:Q0-Q1} provides it for $n = 0, 1$. Let us inductively compute it for $n \ge 2$.

\begin{prop}
\label{prop:Q2}

For $n = 2$, we have
\[
|Q_2(\lambda, \mu)| = (\alpha + \lambda)\beta^{(d-2)d+1}H_0^{d-1},
\]
where
\[
H_x := H_x(\lambda, \mu) = \mu^2 - ((d-2)\lambda + d^m - 2)\mu - ((d-1)\lambda^2 + (d^{m-1}x+d-2)\lambda + d^m - 1).
\]
\end{prop}
\begin{proof}
Let $\alpha' := \alpha(\lambda', \mu')$ and $\beta' := \beta(\lambda', \mu')$ following the definition in Proposition~\ref{prop:Q0-Q1}. Then, by that Proposition and Proposition~\ref{prop:recurrence},
\[
\left|Q_2(\lambda, \mu)\right| = \alpha\beta^{d^2 - 3d + 1}\gamma^{d-1}\left|Q_1(\lambda', \mu')\right| = \alpha\beta^{d^2 - 3d + 1}\gamma^{d-1}(\alpha' + \lambda')\beta'^{d-1}.
\]

We can verify the following relations
\[
\alpha' + \lambda' = \frac{\beta}{\alpha}(\alpha + \lambda), \quad
\beta' = \frac{\beta}{\gamma}H_0.
\]
Therefore,
\[
\left|Q_2(\lambda, \mu)\right| = \alpha\beta^{d^2 - 3d + 1}\gamma^{d-1}\frac{\beta}{\alpha}(\alpha + \lambda)\left(\frac{\beta}{\gamma}H_0\right)^{d-1} =
\]
\[
= (\alpha + \lambda)\beta^{(d-2)d + 1}H_0^{d-1}.
\]
\end{proof}

The motivation for the definition of the polynomials $H_x$ from Proposition~\ref{prop:Q2} will become apparent in Proposition~\ref{prop:factorization}. They form a family of polynomials in $\lambda$ and $\mu$ indexed by the point $x \in \R$. For different values of $x \in \R$, the equation $H_x = 0$ defines different hyperbolas in $\lambda$ and $\mu$.

\begin{prop}
\label{prop:factorization}

For any $n \ge 2$, we have the factorization
\[
\left|Q_n(\lambda, \mu)\right| = (\alpha + \lambda)\beta^{(d-2)d^{n-1} + 1}\prod_{k=0}^{n-2}\prod_{x\in F^{-k}(0)} H_x^{(d-2)d^{n-k-2} + 1},
\]
with $\alpha$ and $\beta$ as in Proposition~\ref{prop:Q0-Q1}, $H_x$ as in Proposition~\ref{prop:Q2} and $F$ being the map
\[
F(x) = x^2 - d(d-1).
\]
\end{prop}
\begin{proof}
The case $n=2$ is shown in Proposition~\ref{prop:Q2}. We use again the recurrence in Proposition~\ref{prop:recurrence} to show the result for $n \ge 3$ inductively. Let $H_x' := H_x(\lambda', \mu')$. We can verify
\[
H'_x = \frac{\beta}{\alpha\gamma}\prod_{y \in F^{-1}(x)}H_y.
\]
Using this relation and the fact that, for any $k \ge 0$, $|F^{-k}(0)| = 2^k$, we have
\[
\left|Q_n(\lambda, \mu)\right| = \left(\alpha\beta^{d^2 - 3d + 1}\gamma^{d-1}\right)^{d^{n-2}}\left|Q_{n-1}(\lambda', \mu')\right| =
\]\[
= \left(\alpha\beta^{d^2 - 3d + 1}\gamma^{d-1}\right)^{d^{n-2}}(\alpha' + \lambda')\beta'^{(d-2)d^{n-2} + 1} \prod_{k=0}^{n-3}\prod_{x \in F^{-k}(0)} H_x'^{(d-2)d^{n-k-3} + 1} =
\]\scriptsize\[
= \left(\alpha\beta^{d^2 - 3d + 1}\gamma^{d-1}\right)^{d^{n-2}}\frac{\beta}{\alpha}(\alpha + \lambda)\left(\frac{\beta}{\gamma}H_0\right)^{(d-2)d^{n-2} + 1}\prod_{k=0}^{n-3}\prod_{x \in F^{-k}(0)} \left(\frac{\beta}{\alpha\gamma}\prod_{y \in F^{-1}(x)}H_y\right)^{(d-2)d^{n-k-3} + 1} =
\]\tiny\[
= (\alpha\gamma)^{d^{n-2}-1}\beta^{(d^2 - 2d - 1)d^{n-2} + 2}(\alpha + \lambda) H_0^{(d-2)d^{n-2} + 1}\prod_{k=0}^{n-3}\left(\frac{\beta}{\alpha\gamma}\right)^{2^k((d-2)d^{n-k-3} + 1)}\prod_{x \in F^{-(k+1)}(0)}H_x^{(d-2)d^{n-k-3} + 1} =
\]\scriptsize\[
= (\alpha\gamma)^{d^{n-2}-1}\beta^{(d^2 - 2d - 1)d^{n-2} + 2}(\alpha + \lambda) H_0^{(d-2)d^{n-2} + 1}\left(\frac{\beta}{\alpha\gamma}\right)^{d^{n-2} - 1}\prod_{k=1}^{n-2}\prod_{x\in F^{-k}(0)}H_x^{(d-2)d^{n-k-2} + 1} =
\]\normalsize\[
= (\alpha + \lambda)\beta^{(d^2 - 2d)d^{n-2} + 1} \prod_{k=0}^{n-2}\prod_{x\in F^{-k}(0)}H_x^{(d-2)d^{n-k-2} + 1} = \]\[ =
(\alpha + \lambda)\beta^{(d - 2)d^{n-1} + 1} \prod_{k=0}^{n-2}\prod_{x\in F^{-k}(0)}H_x^{(d-2)d^{n-k-2} + 1}.
\]
\end{proof}

The relation between the determinants of $Q_n(\lambda, \mu)$ and $Q_{n-1}(\lambda', \mu' )$ is given by the substitution $\lambda \mapsto \lambda'$, $\mu \mapsto \mu'$. For $Q_2$, one of the factors of the determinant is the polynomial we called $H_0$. To compute the determinant of $Q_3$, we have to develop $H_0'$. It is in this analysis that the polynomials $H_x$ and the map $F$ arise. They are the link between $H_x'$ and $H_y$ that allows us to unfold the recurrence.

From the factorization in Proposition~\ref{prop:factorization} we can extract $\spec(M_n)$, as we mentioned above, by setting $\lambda = 1$. Recall that $|S| = d^m + d - 2$.

\begin{thm}
	\label{thm:finite_spec}
	
	We have $\spec(M_0) = \{1\}$, $\spec(M_1) = \{1, \frac{|S| - d}{|S|}\}$ and, for $n \ge 2$,
	\[
	\spec(M_n) = \left\lbrace1, \frac{|S| - d}{|S|}\right\rbrace \bigcup \psi^{-1}\left(\bigcup_{k=0}^{n-2}F^{-k}(0)\right),
	\]
	where $F(x) = x^2 - d(d-1)$ and $\psi(t) = \frac{1}{d^{m-1}}(|S|^2t^2 - |S|(|S|-2)t - (|S| + d-2))$.
\end{thm}
\begin{proof}
	We already established that $\spec(M_n) = \{\frac{\mu}{|S|} \mid \left|Q_n(1, \mu)\right| = 0\}$, and by Proposition~\ref{prop:factorization}, the determinant only vanishes in the following cases:
	\begin{itemize}
		\item $\alpha + 1 = 0 \implies \mu = |S| \implies \frac{\mu}{|S|} = 1$ , with multiplicity $1$.
		\item $\beta = 0 \implies \mu = d^m - 2 \implies \frac{\mu}{|S|} = \frac{|S| - d}{|S|}$, with multiplicity $(d-2)d^{n-1} + 1$.
		\item $H_x = 0$, for some $x\in F^{-k}(0)$ with $0 \le k \le n-2$. This implies that $\mu = \frac{|S| - 2}{2} \pm \sqrt{\left(\frac{d^m + d}{2}\right)^2 - d^m + d^{m-1}x}$, each with multiplicity $(d-2)d^{n-k-2} + 1$. Equivalently, $\frac{\mu}{|S|}$ is one of the two preimages of $x$ by the map $\psi$ defined above.
	\end{itemize}
\end{proof}

\section{Spectra of infinite Schreier graphs}
\label{sec:infinite}

Once we have found $\spec(M_n)$ in Theorem~\ref{thm:finite_spec}, we can prove Theorem~\ref{thm:spectrum_schreier} using the relation $\spec(M_\xi) = \overline{\bigcup\limits_{n\ge0} \spec(M_n)}$, for any $\xi \in X^\N$. In particular, the spectrum does not depend on the choice of $\xi$. The inclusion $\subseteq$ follows from weak containment of representations, see~\cite{dixmier}, Theorem 3.4.9. The equality holds if $\Gamma$ is amenable, which is our case, because our graphs are of polynomial growth (see~\cite{B12}), and hence amenable.

\input{img/parabolas}

\spectrumschreier*
\begin{proof}
	The first statement is an immediate consequence of Theorem~\ref{thm:finite_spec}, as explained above. We just remark that $1$ is obtained as a preimage by $\psi$ of $d$, the limit of the sequence $(F^{-k}_1(0))_k$, where $F_1^{-1}$ is the positive branch of the inverse. For $d = 2$, the map $F$ is $x^2 - 2$, whose Julia set is the interval $[-2, 2]$, and $\psi(t) = 2^{m+1}t^2 - (2^{m+1} - 4)t - 2$. We can find the preimages $t$ of any $x \in [-2, 2]$ by $\psi$:
	\[
	x = 2^{m+1}t^2 - (2^{m+1} - 4)t - 2 \implies 2^{m+1}t^2 - (2^{m+1} - 4)t - (2 + x) = 0 \implies
	\]\[
	\implies t = \frac{2^{m-1} - 1}{2^m} \pm \frac{1}{2^m}\sqrt{4^{m-1} + 1 + 2^{m-1}x}.
	\]
	And so
	\[
	t \in \frac{1}{2^m} \left(2^{m-1} - 1 \pm \sqrt{4^{m-1} + 1 + 2^{m-1}[-2,2]}\right) = \frac{1}{2^m} \left(2^{m-1} - 1 \pm \sqrt{[(2^{m-1} - 1)^2, (2^{m-1} + 1)^2]}\right) = \]\[ = \frac{1}{2^m} \left(2^{m-1} - 1 \pm [2^{m-1} - 1, 2^{m-1} + 1]\right) = \frac{1}{2^m} \left([-2, 0]\cup[2^m - 2, 2^m]\right) = [-\frac{1}{2^{m-1}}, 0]\cup[1 - \frac{1}{2^{m-1}}, 1].
	\]
	
	If $d \ge 3$, then the Julia set of $F$ is a Cantor set of zero Lebesgue measure, and its two preimages by $\psi$ are still Cantor sets lying on the different sides of the minimal point of $\psi$, $\frac{|S| - 2}{2|S|}$.
	
	This completes the proof. Intuitively, we can regard $\spec^0(\Gamma)$ as two infinite trees (one for each branch of $\psi^{-1}$), and $\spec^\infty(\Gamma)$ as their boundary.
\end{proof}

The map $\psi$ is symmetric about its minimal point $v = \frac{|S|-2}{2|S|}$, and satisfies
\[
\psi^{-1}(d) = \left\{1, -\frac{2}{|S|}\right\}, \qquad \psi^{-1}(-d) = \left\{ v \pm \frac{\sqrt{(|S| + 2)^2 - 8(|S| - d+2)}}{2|S|}\right\},
\]\[
\psi^{-1}(-d(d-1)) = \left\{ \frac{|S|-d}{|S|}, \frac{d-2}{|S|}\right\}.
\]
For example, for $m=1$, we have $\psi(t) = 4(d-1)^2t^2 - 4(d-1)(d-2)t - (3d-4)$, and
\[
\spec(M_\xi) = \psi^{-1}\left(\overline{\bigcup_{n \ge 0}F^{-n}(0)}\right),
\]
for any $\xi \in X^\N$, as in this case $\frac{|S|-d}{|S|}$ belongs to the preimage of the Julia set by $\psi$.

In the proof of Theorem~\ref{thm:spectrum_schreier} we actually found the explicit multiplicities of the eigenvalues of $M_n$. Before moving on to study the Kesten spectral measures, let us use these multiplicities to compute the density of states (empirical measure) of the graphs $\{\Gamma_n\}_n$ (see the definition just before Theorem~\ref{thm:KNS_spectral_measure} in the Introduction). It represents the spatial averaging of the Kesten measures.


\input{img/spectra}

\KNSspectralmeasure*
\begin{proof}
	Let $\nu_n$ be the counting measure on the spectrum of $M_n$, i.e.
	\[
	\nu_n = \frac{1}{d^n}\sum_{\lambda \in \spec(M_n)}\delta_\lambda.
	\]
	
	From the multiplicities computed in the proof of Theorem~\ref{thm:finite_spec}, we have that $\nu_0 = \delta_1$, $\nu_2 = \frac{1}{d}(\delta_1 + (d-1)\delta_{\frac{|S| - d}{|S|}})$ and, for $n \ge 2$,
	\[
	\nu_n = \frac{1}{d^n}\left(\delta_1 + \left((d-2)d^{n-1} + 1\right)\delta_{\frac{|S| - d}{|S|}}  + \sum_{k=0}^{n-2}\sum_{x \in F^{-k}(0)}\left((d-2)d^{n-k-2} + 1\right)\left(\delta_{\psi^{-1}_0(x)} + \delta_{\psi^{-1}_1(x)}\right)\right),
	\]
	with $\psi^{-1}_0$ and $\psi^{-1}_1$ being the two branches of the inverse of $\psi$.
	
	For $d > 2$, we observe, in the limit as $n\to\infty$, the measure
	\[
	\nu = \frac{d-2}{d}\delta_{\frac{|S| - d}{|S|}} + \sum_{n \ge 0}\sum_{x \in F^{-n}(0)}\frac{d-2}{d^{n+2}}\left(\delta_{\psi^{-1}_0(x)} + \delta_{\psi^{-1}_1(x)}\right),
	\]
	as in the statement.
	
	For $d = 2$, all the multiplicities of the eigenvalues in the finite graphs are $1$, or equivalently, every eigenvalue of $M_n$ has the same measure $\frac{1}{d^n}$. When taking the limit, the measure of each atom tends to zero and the set of eigenvalues is dense in either one ($m=1$) or two ($m \ge 2$) intervals. However, any set of positive spectral measure still has to be the union of cones of the tree of preimages of $F$ plus their closure, which would have positive Lebesgue measure. Hence, $\nu$ is absolutely continuous with respect to the Lebesgue measure. We can find its precise density if we notice the following, for $d=2$ and $n \ge 1$:
	\[
	\spec(\Gamma_n) = \left\{ \frac{1}{2} - \frac{1}{2^m} + \frac{(-1)^\epsilon}{2^m}\sqrt{4^{m-1} + 1 + 2^m \cos\theta} \bigm| \epsilon \in \{0, 1\}, \theta \in \frac{2\pi\Z}{2^n} \right\} \setminus \left\{0, - \frac{1}{2^{m-1}} \right\}.
	\]

	Indeed, from the proof of Theorem~\ref{thm:finite_spec} we recover the two branches of the inverse of $\psi$:
	\[
	\psi_\epsilon(x) = \frac{1}{2} - \frac{1}{2^m} + \frac{(-1)^\epsilon}{2^m}\sqrt{4^{m-1} + 1 + 2^{m-1}x}.
	\]
	Any $x \in F^{-k}(0)$ can be written as $x = \pm \sqrt{2 + y}$, with $y \in F^{-(k-1)}(0)$. We can hence complete the proof of the equality above by induction, using the trigonometric identity $2\cos(\frac{\theta}{2}) = \pm \sqrt{2 + 2\cos(\theta)}$. This allows us to find an injective, measure-preserving map $\chi: [0,\pi]\times \{0, 1\} \to \R$, defined by
	\[
	\chi(\theta, \epsilon) = \frac{1}{2} - \frac{1}{2^m} + \frac{(-1)^\epsilon}{2^m}\sqrt{4^{m-1} + 1 + 2^m \cos\theta},
	\]
	with the spectrum uniformly distributed on $[0,\pi]\times \{0, 1\}$. The measure of any subset $E \subset \R$ is $\nu(E) = \lambda(\chi^{-1}(E))$, with $\lambda$ being the Lebesgue measure on $[0,\pi]\times \{0, 1\}$. The density $g(x)$ of $\nu$ is thus given by
	\[
	g(x) = \frac{1}{2\pi}\frac{d}{dx}\chi^{-1}(x),
	\]
	which coincides with the expression in the statement.
\end{proof}



	

We end this section by proving that in the case $d = 2$ the Kesten spectral measures for graphs $\Gamma_\xi$ for all $\xi$ except the orbit of $1^\N$ are equal to the density of states. More precisely, we prove the following result.

\begin{center}
	\begin{figure}[ht]
		\includegraphics[scale=1]{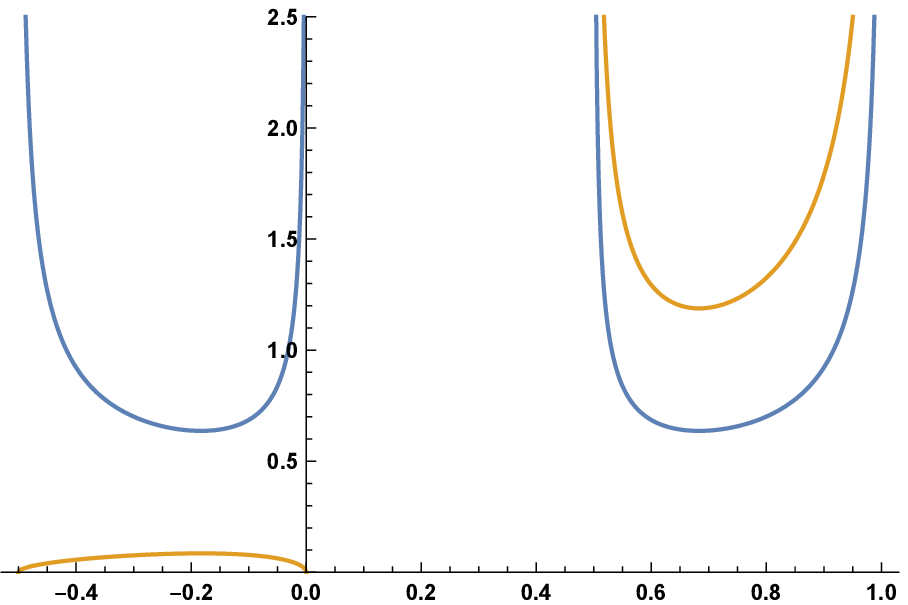}
		\caption{Densities of the spectral measures $\mu_\xi$ for $d=2$ and $m=2$. In blue, the symmetric density corresponds to points not in the orbit of $1^\N$, which have two-ended lines as Schreier graphs. In orange, the asymmetric density corresponds to the point $1^\N$, whose Schreier graph is a one-ended line.}
		\label{fig:spectral-measures}
	\end{figure}
\end{center}

\Kestenspectralmeasurebinary*
\begin{proof}
First recall that for any $\xi$ not in the orbit of $1^\N$ the graphs $\Gamma_\xi$ are two-ended lines. More precisely, every vertex has $2^{m-1} - 1$ loops, $2^m - 2^{m-1}$ edges to one neighbor and one edge to the other neighbor. The simple random walk on such graphs is described by the Markov chain on $\Z$ with probability $\frac{1}{2} - \frac{1}{2^m}$ of staying at any vertex, and alternating probabilities $\frac{1}{2}$ and $\frac{1}{2^m}$ on the other edges. This implies that the Kesten spectral measures $\mu_\xi$ do not depend on this point $\xi$, except for $\xi$ in the orbit of $1^\N$.

The density of states $\nu$ is the integral of the Kesten measures $\mu_\xi$ over all $X^\N$ (see Theorem 10.8 in~\cite{G11}), but we just showed that they are all equal in a subset of $X^\N$ of measure one. Hence, we necessarily have $\mu_\xi = \nu$ for every $\xi$ in that subset.

The density $h(x)$ of $\mu_{1^\N}$ is computed with an approach similar to that in~\cite{GK12}. It uses the fact that the Stieltjes transform of the density of a spectral measure of the Markov operator on a graph coincides with its moment-generating function. We omit the technical computations.
\end{proof}


Recall that fixing $d=2$ and $m$ gives us uncountably many isospectral groups. Moreover, for those groups, Proposition~\ref{prop:Kesten_spectral_measure_binary} concludes that, for a subset of boundary points of measure one, the Kesten spectral measures on the orbital Schreier graphs coincide. It would be very interesting to determine the Kesten spectral measures on the Cayley graphs of these groups.

\section{Pure point spectrum and eigenfunctions}
\label{sec:pure_point_eigenfunctions}

This section is devoted to the proof of Theorem~\ref{thm:Kesten_spectral_measure}. We will establish that, for $d \ge 3$, the Kesten spectral measures on $\Gamma_\xi$ are discrete and the eigenfunctions of the Markov operator $M_\xi$ are finitely supported, for every $\xi$ in a set of uniform Bernoulli measure one. To do that, we will use the following strategy. We will first find the eigenfunctions on the finite graphs $\Gamma_n$ (Proposition~\ref{prop:eigenfunctions_finite} and Corollary~\ref{cor:basis_Mn}). Next, we will extend those to $\Gamma_\xi$ and show that some of them remain eigenfunctions (Theorem~\ref{thm:eigenfunctions_Mxi}). Finally, we will show that the set $\F$ of eigenfunctions that we constructed is complete for every $\xi$ in a set of uniform Bernoulli measure one. As $\F$ is a complete set of eigenfunctions for $M_\xi$, any spectral measure $\mu_f$ of $M_\xi$ associated to $f \in \ell^2(\Gamma_\xi)$ must be discrete, in particular this holds for the Kesten spectral measures $\mu_\eta$, $\eta \in \text{Vert}(\Gamma_\xi)$. Moreover we show that all functions in $\F$ are of finite support.



In this section we assume $d \ge 3$. Let us write $\elln = \ell^2(V(\Gamma_n))$ and $\ell^2 = \ell^2(V(\Gamma_\xi))$. We start by defining a notion of antisymmetry on the graphs $\Gamma_n$, which will be satisfied by the eigenfunctions. Let $\tau_i = (i, i+1) \in Sym(X)$, for $i \in \{0, \dots, d-2 \}$, and let $\Phiin: \Gamma_n \to \Gamma_n$ be the automorphisms of $\Gamma_n$ defined by $\Phiin(v_0 \dots v_{n-1}) = v_0 \dots v_{n-2}\tau_i(v_{n-1})$. Recall that the graph $\Gamma_n$ can be decomposed as $d$ copies of $\Gamma_{n-1}$ each of which is connected to the others only through one vertex. $\Phiin$ exchanges the $i$-th and ($i+1$)-th copies of $\Gamma_{n-1}$ in this decomposition. We will say that $f \in \elln$ is antisymmetric with respect to $\Phiin$ if $f = - f \circ \Phiin$. In particular, this implies that $f$ is supported only on the $i$-th and ($i+1$)-th copies of $\Gamma_{n-1}$ in $\Gamma_n$.


\begin{prop}
	\label{prop:basis_MN}
	Let $N \ge 1$ and $\lambda \in \spec(M_N) \setminus \spec(M_{N-1})$. There is a basis $\F_{\lambda, N} = \{f_0, \dots, f_{d-2}\}$ of the $\lambda$-eigenspace of $M_N$, such that $f_i$ is antisymmetric with respect to $\Phiin$, for every $i\in\{0, \dots, d-2\}$. In particular, each $f_i$ is supported in $X^{N-1}i \sqcup X^{N-1}(i+1)$.
\end{prop}
\begin{proof}
	We know that the multiplicity of $\lambda$ in $\spec(M_N)$ is exactly $d-1$, as we computed in the proof of Theorem~\ref{thm:spectrum_schreier}.

	Due to the symmetry of $\Gamma_N$, given a $\lambda$-eigenfuction $f \in \elln$, we know that $f_i = f - f \circ \Phiin$ will be antisymmetric with respect to $\Phiin$ and will still be a $\lambda$-eigenfunction, for any $i\in\{0, \dots, d-2\}$. Furthermore, the fact that these functions are linearly independent becomes clear upon examination of their supports.
\end{proof}

Now, using the notations of Proposition~\ref{prop:basis_MN}, we partition the basis $\F_{\lambda, N}$ into four parts in order to obtain the eigenfunctions of $M_n$, for $n \ge N$:
\[
\F_{\lambda, N}^A := \{f_{d-2}\},
\qquad
\F_{\lambda, N}^B := \{f_0\},
\]\[
\F_{\lambda, N}^C := \F_{\lambda, N} \setminus \{f_0, f_{d-2}\},
\qquad
\F_{\lambda, N}^D := \emptyset.
\]

We would like to translate these eigenfunctions from $\Gamma_N$ to $\Gamma_n$, with $n \ge N$. Recall that the graph $\Gamma_{n+1}$ consists of $d$ copies of $\Gamma_n$ joined together by a central piece. We will take advantage of this decomposition with the following natural graph inclusions. Let $n \ge 1$ and $i \in X$. We define
\[
\iota^i_n: \Gamma_n \to \Gamma_{n+1}, \quad \iota^i_n(v) = vi.
\]

We may also define the following induced linear operator (see Figure~\ref{fig:rhos}):
\[
\rho^i_n: \elln \to \ellnn, \quad \rho^i_nf(vj) = \begin{cases}
f\circ(\iota^i_n)^{-1}(vj) & \textrm{if } i=j \\ 0 & \textrm{otherwise}
\end{cases}.
\]
Equivalently, $\rho^i_nf(vj) = f(v)\delta_{i,j}$, where $\delta_{i,j}$ is $1$ if $i=j$ or $0$ otherwise. Let also $\rho^n = \sum_{i \in X} \rho_i^n$.

\input{img/rhos}

Now, in order to get the eigenfunctions of $M_{n+1}$ from those of $M_n$, we apply these transition functions $\rho^i_n$ in the following way, for $n \ge N$,
\[
\F_{\lambda, n+1}^A := \rho^{d-1}_n(\F_{\lambda, n}^A),
\qquad
\F_{\lambda, n+1}^B := \rho^0_n(\F_{\lambda, n}^A),
\]\[
\F_{\lambda, n+1}^C := \bigsqcup_{i \neq 0, d-1} \rho^i_n(\F_{\lambda, n}^A) \sqcup \rho_n(\F_{\lambda, n}^B),
\qquad
\F_{\lambda, n+1}^D := \bigsqcup_{i \in X} \rho^i_n(\F_{\lambda, n}^C \sqcup \F_{\lambda, n}^D).
\]

Finally, we set $\F_{\lambda, n} := \F_{\lambda, n}^A \sqcup \F_{\lambda, n}^B \sqcup \F_{\lambda, n}^C \sqcup \F_{\lambda, n}^D$. 

\begin{rmk}
	One can look at the supports of the functions in $\F_{\lambda, n}^A$, $\F_{\lambda, n}^B$, $\F_{\lambda, n}^C$ and $\F_{\lambda, n}^D$ to verify that these four sets are disjoint, and the following statements can be inductively proven:
	\begin{itemize}
		\item $|\F_{\lambda, n}^A| = |\F_{\lambda, n}^B| = 1$, $\forall n \ge N$.
		\item $|\F_{\lambda, N}^C| = d-3$, and $|\F_{\lambda, n}^C| = d-1$, $\forall n \ge N+1$.
		\item $|\F_{\lambda, n}| = (d-2)d^{n-N} + 1$, $\forall n \ge N$.
	\end{itemize}
	The sizes of $\F_{\lambda, n}^A$, $\F_{\lambda, n}^B$ and $\F_{\lambda, n}^C$ are uniformly bounded for all $n$. However, the size of $\F_{\lambda, n}^D$ grows with $n$. Furthermore, notice that, by construction, the following statements hold for every $n \ge N$:
	\[
	\forall f \in \F_{\lambda, n}\setminus\F_{\lambda, n}^B, \quad f((d-1)^{n-1}0) = 0,
	\]\[
	\forall f \in \F_{\lambda, n}\setminus \F_{\lambda, n}^A, \quad f((d-1)^n) = 0.
	\]
\end{rmk}

\begin{prop}
	\label{prop:eigenfunctions_finite}
	Let $N \ge 1$ and $\lambda \in \spec(M_N) \setminus \spec(M_{N-1})$. Then $\F_{\lambda, n}$ is a basis of the $\lambda$-eigenspace of $M_n$, for every $n \ge N$.
\end{prop}
\begin{proof}
	Let us proceed by induction on $n$, with the base case $n=N$ covered in Proposition~\ref{prop:basis_MN}.
	
	Let $f \in \F_{\lambda, n}$ be a $\lambda$-eigenfunction of $M_n$. and let $v \in X^n$, $j \in X$ and $s \in S$. On one hand we have
	\[
	\rho^i_nM_nf(vj) = M_nf(v) \delta_{i,j} = \frac{1}{|S|}\sum_{s \in S} f(s(v)) \delta_{i,j}.
	\]
	On the other hand, if $v \neq (d-1)^{n-1}0$, we have $s(vj) = s(v)j$. In that case,
	\[
	M_{n+1}\rho^i_nf(vj) = \frac{1}{|S|}\sum_{s \in S} \rho^i_nf(s(vj)) = \frac{1}{|S|}\sum_{s \in S} \rho^i_nf(s(v)j) = \frac{1}{|S|}\sum_{s \in S} f(s(v)) \delta_{i,j}.
	\]
	So we have $M_{n+1}\rho^i_nf(vj) = \rho^i_nM_nf(vj) = \lambda \rho^i_nf(vj)$ if $v \neq (d-1)^{n-1}0$.
	
	For $v=(d-1)^{n-1}0$, we need to further decompose the sums:
	\[
	\rho^i_nM_nf(vj) = M_nf(v)\delta_{i,j} =
	\]\[
	= \frac{1}{|S|}\left( \sum_{k = 1}^{d-1}f(a^k(v))\delta_{i,j} + \sum_{1 \neq b \in B} f(b(v))\delta_{i,j}\right) =
	\]\[
	= \frac{1}{|S|}\left( \sum_{k = 1}^{d-1}f(a^k(v))\delta_{i,j} + \sum_{1 \neq b \in B} f(v)\delta_{i,j}\right),
	\]
	since $v$ is fixed by all $b\in B$, and
	\[
	M_{n+1}\rho^i_nf(vj) =
	\]\[
	= \frac{1}{|S|}\left( \sum_{k = 1}^{d-1}\rho^i_nf(a^k(vj)) + \sum_{1 \neq b \in B} \rho^i_nf(b(vj))\right) =
	\]\[
	= \frac{1}{|S|}\left( \sum_{k = 1}^{d-1}\rho^i_nf(a^k(v)j) + \sum_{1 \neq b \in B} \rho^i_nf(v \omega_{n-1}(b)(j))\right) =
	\]\[
	= \frac{1}{|S|}\left( \sum_{k = 1}^{d-1}f(a^k(v))\delta_{i,j} + \sum_{1 \neq b \in B} f(v)\delta_{i,\omega_{n-1}(b)(j)}\right).
	\]
	
	By substracting both expressions, we get
	\[
	(M_{n+1} - \lambda)\rho^i_nf(vj) = \frac{1}{|S|}\left( \sum_{1 \neq b \in B} f(v)(\delta_{i,\omega_{n-1}(b)(j)} - \delta_{i,j})\right).
	\]
	
	We observe now that if $f \in \F_{\lambda, n}\setminus \F_{\lambda, n}^B$, by construction, we have $f(v) = 0$, and so $\rho^i_nf$ is a $\lambda$-eigenfunction of $M_{n+1}$. Else, if $f \in \F_{\lambda, n}^B$, we add the equations for all $i \in X$:
	\[
	(M_{n+1} - \lambda)\rho_nf(vj) = \frac{1}{|S|}\left( \sum_{1 \neq b \in B} f(v)\sum_{i \in X}(\delta_{i,\omega_{n-1}(b)(j)} - \delta_{i,j})\right) = 0.
	\]
	In this case, $\rho_nf$ is an eigenfunction of $M_{n+1}$.
	
	We can inductively verify that the functions in $\F_{\lambda, n+1}$ are linearly independent by looking at the supports of the images of the functions from $\F_{\lambda, n}$ by $\rho^i_n$ and $\rho_n$. Finally, we already know that $|\F_{\lambda, n+1}| = (d-2)d^{n+1-N} + 1$, which equals the multiplicity of $\lambda$ for $M_{n+1}$, and so the dimension of the $\lambda$-eigenspace of $M_{n+1}$.
\end{proof}

Set $\F_{1, n}$ to be the singleton containing the constant function equal to one on $\Gamma_n$, $n \ge 0$.


\begin{cor}
	\label{cor:basis_Mn}
	The set
	\[
	\bigsqcup_{\lambda \in \spec(M_n)}  \F_{\lambda, n}
	\]
	is a basis of $\elln$ that consists of eigenfunctions of $M_n$.
\end{cor}

We now want to describe the eigenfunctions of $M_\xi$, the Markov operator on the infinite graph $\Gamma_\xi$, whose vertex set is $G\xi$ (the cofinality class of $\xi$ in $X^\N$). We will do so by translating the eigenfunctions in the finite graphs $\Gamma_n$ to $\Gamma_\xi$ via the transfer operators we now define. Let $\tilde{\iota}_n$ be the canonical extension of the inclusions $\iota^i_n$ for the finite graphs $\Gamma_n$ to the infinite graph $\Gamma_\xi$. Namely, $\tilde{\iota}_n: \Gamma_n \to \Gamma_\xi$, defined by $\tilde{\iota}_n(v) = v\sigma^n(\xi)$. In addition, we define the following natural operators linking functions on $\Gamma_n$ with functions on $\Gamma_\xi$. For $n \ge 0$, we set
\[
\tilde{\rho}_n: \elln \to \ell^2, \quad \tilde{\rho}_nf(\eta) = \begin{cases}
f\circ\tilde{\iota}_n^{-1}(\eta) & \textrm{if } \sigma^n(\xi) = \sigma^n(\eta) \\ 0 & \textrm{otherwise}
\end{cases}.
\]
Equivalently, $\tilde{\rho}_nf(\eta) = f(\eta_0\dots\eta_{n-1})\delta_{\sigma^n(\xi), \sigma^n(\eta)}$, where again $\delta_{a, b}$ is one iff $a=b$, zero otherwise. Informally, we could write $\tilde{\iota}_n = \dots \circ \iota^{\xi_{n+1}}_{n+1} \circ \iota^{\xi_n}_n$. Let us then define the following set:
\[
\F_\lambda := 
\bigcup_{n \ge N} \tilde{\rho}_n(\F_{\lambda, n}^D).
\]

If there exists $r \ge 0$ such that $\sigma^r(\xi) = (d-1)^\N$, let it be minimal and set $R = \max\{r, N\}$. In that case, we also include the function $\tilde{\rho}_R(\F_{\lambda, R}^A)$ in the definition of $\F_\lambda$.

\begin{thm}
	\label{thm:eigenfunctions_Mxi}
	Let $N \ge 1$ and $\lambda \in \spec(M_N) \setminus \spec(M_{N-1})$. Then every $f\in\F_\lambda$ is a $\lambda$-eigenfunction of $M_\xi$, for every $\xi \in X^\N$.
\end{thm}
\begin{proof}
	Let $n \ge N$ and $f \in \F_{\lambda, n}^D$. We will show that $\tilde{\rho}_n(f)$ is a $\lambda$-eigenfunction of $M_\xi$. Let $\eta \in G\xi$ and denote by $v$ its prefix of length $n$, so that $\eta = v \sigma^n(\eta)$. Assume first that $v$ is not $(d-1)^{n-1}0$ nor $(d-1)^n$. In that case, for any $s \in S$, $s(\eta) = s(v \sigma^n(\eta)) = s(v)\sigma^n(\eta)$. On one hand,
	\[
	M_{\xi}\tilde{\rho}_nf(\eta) = \frac{1}{|S|}\sum_{s \in S} \tilde{\rho}_nf (s(\eta)) = \frac{1}{|S|}\sum_{s \in S} \tilde{\rho}_nf (s(v)\sigma^n(\eta)) =
	\]\[
	= \frac{1}{|S|}\sum_{s \in S} f(s(v))\delta_{\sigma^n(\xi),\sigma^n(\eta)}.
	\]

	On the other hand,
	\[
	\tilde{\rho}_nM_nf(\eta) = M_nf(v)\delta_{\sigma^n(\xi), \sigma^n(\eta)} = \frac{1}{|S|}\sum_{s \in S} f(s(v))\delta_{\sigma^n(\xi),\sigma^n(\eta)}.
	\]
	
	We observe that both expressions are equal. Therefore,
	\[
	M_\xi \tilde{\rho}_nf(\eta) = \tilde{\rho}_nM_nf(\eta) = \lambda \tilde{\rho}_nf(\eta).
	\]
	
	Now let $v = (d-1)^{n-1}0$ or $v = (d-1)^n$. In that case, as $f \in \F_{\lambda, n}^D$, we have $f(v) = 0$. Then,
	\[
	M_\xi \tilde{\rho}_nf(\eta) = \frac{1}{|S|}\sum_{s \in S} \tilde{\rho}_nf (s(\eta)) = \frac{1}{|S|}\sum_{s \in S} \tilde{\rho}_nf (s(v\sigma^n(\eta))) =
	\]\[
	= \frac{1}{|S|}\sum_{s \in S} \tilde{\rho}_nf (s(v)s_{v}(\sigma^n(\eta))) = \frac{1}{|S|}\sum_{s \in S} f(s(v)) \delta_{\sigma^n(\xi), s_{v}(\sigma^n(\eta))}.
	\]
	
	In addition,
	\[
	\tilde{\rho}_nM_nf(\eta) = M_nf(v)\delta_{\sigma^n(\xi), \sigma^n(\eta)} = \frac{1}{|S|}\sum_{s \in S} f(s(v))\delta_{\sigma^n(\xi),\sigma^n(\eta)}.
	\]

	We have two cases: either $s \in A$, which means that $s_v$ is trivial, or $s \in B$, so $s(v) = v$ and then $f(s(v)) = 0$. In any case, the two expressions above coincide. Consequently, $M_\xi \tilde{\rho}_nf(\eta) = \tilde{\rho}_nM_nf(\eta) = \lambda \tilde{\rho}_nf(\eta)$ as well for $v = (d-1)^{n-1}0$ and $v = (d-1)^n$, which shows that $f$ is a $\lambda$-eigenfunction of $M_\xi$.
	
	The case $\xi$ cofinal with $(d-1)^\N$ and $f \in \F_{\lambda, R}^A$ is proven in a very similar way. The only difference is that now $f(v)$ is not necessarily zero for $v = (d-1)^R$. However, for any $s \in B$, $\delta_{\sigma^n(\xi), s_v(\sigma^n(\eta))} = \delta_{s_v^{-1}(\sigma^n(\xi)), \sigma^n(\eta)} =  \delta_{\sigma^n(\xi), \sigma^n(\eta)}$, since $\sigma^n(\xi) = (d-1)^\N$ is fixed by $s_v^{-1}$. Therefore, both expressions are still equal and the statement remains true for this case too.
\end{proof}

We finally introduce the set
\[
\F := \bigcup_{N \ge 1}\bigcup_{\substack{\lambda \in \spec(M_N) \\ \lambda \not\in \spec(M_{N-1})}} \F_\lambda.
\]

Notice that if $\lambda \in \spec(M_N)\setminus\spec(M_{N-1})$, the size of the support of any $f \in \F_\lambda$ is either $2d^{N-1}$ or $2d^N$.

\input{img/gupta-fabrykowski-fups-1}
\input{img/gupta-fabrykowski-fups-2}

We conclude the section with the proof of Theorem~\ref{thm:Kesten_spectral_measure}. For that, we need to extend the notion of antisymmetric functions to $\Gamma_\xi$, in a way similar to what is done in~\cite{BGJRT19}. We define the space of antisymmetric functions on $\Gamma_n$ by
\[
\ellan = \langle f \in \elln \mid \exists i \in \{1, \dots, d-2\},\quad f = -f \circ \Phi^i_n \rangle.
\]


\begin{lem}
\label{lem:ellan}
$\forall n \ge 1$,
	\[
	\ellan = \left\langle \bigsqcup_{N=1}^{n-1} \bigsqcup_{\substack{\lambda \in \spec(M_N) \\ \lambda \not\in \spec(M_{N-1})}} \bigsqcup_{i=1}^{d-2} (\rho^{i+1}_{n-1} - \rho^i_{n-1})(\F_{\lambda, n-1}\setminus \F_{\lambda, n-1}^B) \quad \sqcup \quad \bigsqcup_{\substack{\lambda \in \spec(M_n) \\ \lambda \not\in \spec(M_{n-1})}} \F_{\lambda, n}\setminus \F_{\lambda, n}^B \right\rangle.
	\]
\end{lem}
\begin{proof}
	First notice that all the functions on the right hand side belong to $\ellan$ by construction. Indeed, let $1 \le N \le n-1$ and $\lambda\in \spec(M_N)\setminus\spec(M_{N-1})$. Let also $f \in \F_{\lambda, n-1}\setminus \F_{\lambda, n-1}^B$ and $1 \le i \le d-2$. On one hand we have
	\[
	(\rho^{i+1}_{n-1} - \rho^i_{n-1})f \circ \Phi^i_n (v_0\dots v_{n-1}) = (\rho^{i+1}_{n-1} - \rho^i_{n-1})f (v_0\dots v_{n-2}\tau_i(v_{n-1})) =
	\]\[
	= f(v_0 \dots v_{n-2})(\delta_{i+1, \tau_i(v_{n-1})} - \delta_{i, \tau_i(v_{n-1})}). 
	\]
	On the other hand,
	\[
	(\rho^{i+1}_{n-1} - \rho^i_{n-1})f (v_0\dots v_{n-1}) = f(v_0 \dots v_{n-2})(\delta_{i+1, v_{n-1}} - \delta_{i, v_{n-1}}).
	\]
	If $v_{n-1} \neq i, i+1$, then $\tau_i(v_{n-1}) = v_{n-1}$ and so both expressions vanish. Otherwise, $\tau_i$ exchanges $i$ and $i+1$ and so the latter equals the former with opposite sign.

	If we now take $\lambda \in \spec(M_n)\setminus\spec(M_{n-1})$ and $f \in \F_{\lambda, n}\setminus \F_{\lambda, n}^B$, the fact that $f$ is antisymmetric follows from Proposition~\ref{prop:basis_MN}.

	Finally, we will check that the dimension of both subspaces is the same. We know that $\dim(\ellan) = (d-2)d^{n-1}$ and that the functions on the right hand side are linearly independent, again inductively and using the fact that we know their supports. Their number is:
	\[
	\sum_{N=1}^{n-1} \sum_{\substack{\lambda \in \spec(M_N) \\ \lambda \not\in \spec(M_{N-1})}} \sum_{i=1}^{d-2} |\F_{\lambda, n-1}\setminus \F_{\lambda, n-1}^B| \quad + \quad \sum_{\substack{\lambda \in \spec(M_n) \\ \lambda \not\in \spec(M_{n-1})}} |\F_{\lambda, n}\setminus \F_{\lambda, n}^B| =
	\]\[
	= \sum_{N=1}^{n-1} \sum_{\substack{\lambda \in \spec(M_N) \\ \lambda \not\in \spec(M_{N-1})}} \sum_{i=1}^{d-2} (d-2)d^{n-1-N} \quad + \quad \sum_{\substack{\lambda \in \spec(M_n) \\ \lambda \not\in \spec(M_{n-1})}} (d-2) =
	\]\[
	= (d-2)^2\sum_{N=1}^{n-1} d^{n-1-N} |\spec(M_N)\setminus\spec(M_{N-1})|  \quad + \quad (d-2) |\spec(M_n)\setminus\spec(M_{n-1})| =
	\]\[
	= (d-2)^2\sum_{N=1}^{n-1} d^{n-1-N} 2^{N-1} + (d-2) 2^{n-1} =
	\]\[
	= (d-2)(d^{n-1} - 2^{n-1}) + (d-2) 2^{n-1} =
	\]\[
	= (d-2)d^{n-1} = \dim(\ellan).
	\]
\end{proof}

Let $P_n: \ell^2 \to \ell^2$ be the orthogonal projector to the subspace of functions supported in $\tilde{\Gamma}_n := \tilde{\iota}_n(\Gamma_n) = X^n\sigma^n(\xi)$, so that $P_nf(\eta) = f(\eta) \chi_{\tilde{\Gamma}_n}(\eta)$. Let also $P_n': \ell^2 \to \elln$ be the operator defined by $P_n'f = f \circ \tilde{\iota}_n$. In addition, we define the space of antisymmetric functions on $\Gamma_\xi$ as follows:
\[
\ella = \langle f \in \ell^2 \mid \exists n\in I_\xi \quad \supp(f) \subset \tilde{\Gamma}_n, \quad P_n'f \in \ellan \rangle,
\]
where $I_\xi = \{ n \in \N \mid \forall r \ge 0, \: (d-1)^r0 \textrm{ is not a prefix of } \sigma^n(\xi) \}$ is the set of indices $n$ such that $\sigma^n(\xi)$ has trivial $B$-action. Equivalently, it is the set of indices for which $\tilde{\Gamma}_n$ is connected to the rest of $\Gamma_\xi$ by just one vertex.

\begin{lem}
	\label{lem:ella_in_spanF}
	$\ella$ is contained in $\langle\F\rangle$.
\end{lem}
\begin{proof}
	Let $f \in \ella$. Then there exists some $n \in I_\xi$ such that $\supp(f) \subset \tilde{\Gamma}_n$ and $P_n'f \in \ellan$. In particular, either there exists some $r \ge 0$ such that $\xi_n\dots \xi_{n+r} = (d-1)^rj$, with $j \neq 0, d-1$, or $\sigma^n(\xi) = (d-1)^\N$, in which case we set $r = \infty$. Let $T_n$ be the basis of $\ellan$ from Lemma~\ref{lem:ellan}. We first claim that for every $h \in T_n$, $\tilde{\rho}_nh \in \langle\F\rangle$.

	Indeed, let $1 \le N \le n-1$, $\lambda \in \spec(M_N)\setminus\spec(M_{N-1})$ and $i \in \{1, \dots, d-1\}$. Given a function $g \in \F_{\lambda, n-1}\setminus\F_{\lambda, n-1}^B$, since $i \neq 0$, we have $\rho^i_{n-1}g \in \F_{\lambda, n}\setminus\F_{\lambda, n}^B$. If $r=\infty$, then $\tilde{\rho}_n\rho^i_{n-1}g \in \F$ directly. Otherwise, then as $\xi_n\dots \xi_{n+r} = (d-1)^rj$, we have $\rho^{\xi_{n+r}}_{n+r}\dots\rho^{\xi_n}_n\rho^i_{n-1}g \in \F_{\lambda, n+r+1}^C\cup\F_{\lambda, n+r+1}^D$, and so $\rho^{\xi_{n+r+1}}_{n+r+1}\dots\rho^{\xi_n}_n\rho^i_{n-1}g \in \F_{\lambda, n+r+2}^D$, independently of the value of $\xi_{n+r+1}$. This means that $\tilde{\rho}_n\rho^i_{n-1}g = \tilde{\rho}_{n+r+2}\rho^{\xi_{n+r+1}}_{n+r+1}\rho^{\xi_n+r}_{n+r}\rho^i_{n-1}g \in \F$. Finally, for any generator $h \in T_n$ of the form $h = (\rho^{i+1}_{n-1} - \rho^i_{n-1})g = \rho^{i+1}_{n-1}g - \rho^i_{n-1}g$, we showed that $\tilde{\rho}_nh \in \langle\F\rangle$.

	Similarly, let $\lambda \in \spec(M_n)\setminus\spec(M_{n-1})$, and $g \in \F_{\lambda, n}\setminus\F_{\lambda, n}^B$. If $r = \infty$, then $\tilde{\rho}_ng \in \F$, and otherwise we have $\rho^{\xi_{n+r}}_{n+r}\dots\rho^{\xi_n}_ng \in \F_{\lambda, n+r+1}^C\cup\F_{\lambda, n+r+1}^D$, provided $\xi_n\dots \xi_{n+r} = (d-1)^rj$. Consequently, $\rho^{\xi_{n+r+1}}_{n+r+1}\dots\rho^{\xi_n}_ng \in \F_{\lambda, n+r+2}^D$, again independently of $\xi_{n+r+1}$. Hence, $\tilde{\rho}_ng = \tilde{\rho}_{n+r+2}\rho^{\xi_{n+r+1}}_{n+r+1}\dots\rho^{\xi_n}_ng \in \F$. Now, for every generator $h$ of $T_n$, of the form $h=g$, $\tilde{\rho}_nh$ must also be in $\langle\F\rangle$.

	To conclude, since $P_n'f \in \ellan$, let $P_n'f = \sum_i c_i h_i$, with $h_i \in T$. The support of $f$ is contained in $\tilde{\Gamma}_n$, so $f = \tilde{\rho}_nP_n'f = \sum_i c_i \tilde{\rho}_nh_i \in \langle\F\rangle$, and hence $\ella$ is contained in $\langle \F \rangle$.
\end{proof}


Our next step is to show that $\ella$ is dense in $\ell^2$. However, we are only able to do this under some extra conditions on $\xi \in X^\N$, which fortunately define a subset $W$ of uniform Bernoulli measure one in $X^\N$. Observe that the antisymmetric subspace $\ella$ is of infinite dimension iff the set $I_\xi = \{ n \in \N \mid \forall r \ge 0, \: (d-1)^r0 \textrm{ is not a prefix of } \sigma^n(\xi) \}$ is infinite. Equivalently, iff $\Gamma_\xi$ is one-ended.
For the proof of Theorem~\ref{thm:Kesten_spectral_measure}, we will in fact need a slightly stronger condition than $\Gamma_\xi$ being one-ended. We will need not only that $I_\xi$ is infinite, but also that it contains consecutive pairs $k$ and $k+1$ infinitely often. Let us consider the subset $W \subset X^\N$ defined as $W = \{\xi \in X^\N \mid k, k+1 \in I_\xi \textrm{ for infinitely many } k \}$. Note that this set only depends on $d$, and does not depend on $m$ nor on $\omega \in \Omega_{d, m}$.

\Kestenspectralmeasure*
\begin{proof}
	The set $W$ can be rewritten as
	\[
	W = \{ \xi \in X^\N \mid \forall l \ge 0,\: \exists k\ge l,\: k, k+1\in I_\xi\},
	\]
	and so its complement is $X^*Z$, with
	\[
	Z = \{ \xi \in X^\N \mid \forall k \ge 0,\: k \not\in I_\xi \textrm{ or } k+1 \not\in I_\xi\}.
	\]

	Notice that if, for some $k \in \N$, $\xi_k \neq 0, d-1$, then $k \in I_\xi$. Equivalently, for any $k \not\in I_\xi$, then necessarily $\xi_k = 0, d-1$. Therefore, for any point $\xi \in Z$, either $0 \not\in I_\xi$ or $1 \not\in I_\xi$, which implies that at least one of $\xi_0$, $\xi_1$ is $0, d-1$ (or both). Hence,
	\[
	\mu(Z) \le \left(1 - \left(\frac{d-2}{d}\right)^2\right)\mu(Z) \Longrightarrow \mu(Z) = 0.
	\]
	Finally, $\mu(W) = 1 - \mu(X^\N\setminus W) = 1 - \mu(X^*Z) = 1 - \mu(Z) = 1$.

	We will now prove that the set $\F$ is complete for every $\xi \in W$. Let $f\in \ell^2$ such that $f\perp \ella$. Now let $n$ be the smallest such that $\normell{P_nf} \ge \frac{4}{5}\normell{f}$ and both $n, n+1 \in I_\xi$, which exists as $\xi \in W$. Our goal is to define a function approximating $P_nf$ and antisymmetric. Because $P_nf$ concentrates the major part of the norm of $f$, and $f \perp \ella$, this will allow us to conclude that $f$ must be zero.

	Let $i \in \{1, \dots, d-2\}$ such that $\xi_n \in \{i, i+1\}$. Define $g := \rho^{\xi_n}_nP_n'f - \rho^{\xi_n}_nP_n'f \circ \Phi^i_{n+1} \in \ellann \subset \ellnn$ and also $h := \tilde{\rho}_{n+1}g \in \ell^2$. In fact, by construction, $h \in \ella$. Then,
	\[
	0 = \langle f, h \rangle_{\ell^2} = \langle P_{n+1}'f, g \rangle_\ellnn =
	\]\[
	= \langle P_{n+1}'f,\rho^{\xi_n}_nP_n'f \rangle_\ellnn - \langle P_{n+1}'f, \rho^{\xi_n}_nP_n'f \circ \Phi^i_{n+1} \rangle_\ellnn = 
	\]\[
	= \normellnn{\rho^{\xi_n}_nP_n'f}^2 - \langle P_{n+1}'f\circ \Phi^i_{n+1}, \rho^{\xi_n}_nP_n'f \rangle_\ellnn =
	\]\[
	= \normell{P_nf}^2 - \langle Q_nf, P_nf \rangle_{\ell^2},
	\]
	where $Q_nf(\eta) = f(\eta_0\dots\eta_{n-1}\tau_i(\eta_n)) \delta_{\sigma^n(\eta), \sigma^n(\xi)}$. Notice that this function is supported in $\tilde{\Gamma}_n$ and its values are those of $f$ on the subgraph $\tilde{\iota}_{n+1}\iota^{\tau_i(\xi_n)}_n(\Gamma_n)$. Therefore, its norm is not greater than the norm of $f - P_nf$, supported in $\Gamma_\xi\setminus\tilde{\Gamma}_n$, so $\normell{Q_nf} \le \normell{f - P_nf}$. Now we have, using the Cauchy-Schwarz inequality, 
	\[
	0 = \normell{P_nf}^2 - \langle Q_nf, P_nf \rangle_{\ell^2} \ge
	\]\[
	\ge \frac{4^2}{5^2}\normell{f}^2 - \normell{Q_nf}\normell{P_nf} \ge
	\]\[
	\ge \frac{4^2}{5^2}\normell{f}^2 - \normell{f - P_nf}\normell{f} \ge
	\]\[
	\ge \frac{4^2}{5^2}\normell{f}^2 - \sqrt{1-\frac{4^2}{5^2}}\normell{f}^2 =
	\]\[
	\ge \left(\frac{16}{25} - \frac{3}{5}\right)\normell{f}^2 = \frac{1}{25}\normell{f}^2.
	\]
	Hence $f = 0$.
\end{proof}

\section{Spectra of Cayley graphs}
\label{sec:spectra_cayley}

Any spinal group $G$ with $d=2$ has subexponential growth. This fact is trivial for $m=1$ and is proven in~\cite{Gri84} for $m=2$, and essentially the same proof extends for $m \ge 3$. Therefore any spinal group $G$ with $d = 2$ is amenable. In that case, it is known that $\spec(M_\xi) \subset \spec(G)$ for every $\xi \in X^\N$. Our goal now is to prove Theorem~\ref{thm:spectrum_cayley_d2}. The proof we provide is a modification of the proof given in~\cite{GD17}, where one of the directions of the proof uses a version of Hulanicki Theorem for graphs. For us it is enough to use the classical Hulanicki Theorem, which we now recall.

\begin{thm}[Hulanicki's Theorem]
	Let $G$ be a locally compact group, and let $\lambda_G$ be its left-regular representation. $G$ is amenable if and only if $\lambda_G$ weakly contains any other unitary representation of $G$.
\end{thm}

For any group $G$, its left-regular representation $\lambda_G$ in $\ell^2(G)$ can be extended to the representation of the group algebra $\C[G]$ by bounded operators, setting, for every $t = \sum_{g\in G} c_g g \in \C[G]$, $\lambda_G(t) = \sum_{g\in G} c_g \lambda_G(g)$.

A unitary representation $\rho$ is \emph{weakly contained} in a unitary representation $\eta$ of $G$ (denoted $\rho \prec \eta$) if there exists a surjective homomorphism $C_\rho^* \twoheadrightarrow C_\eta^*$ mapping the operator $\eta(g)$ to $\rho(g)$, for every $g \in G$, where $C_\rho^*$ is the $C^*$-algebra generated by $\rho$. In~\cite{dixmier}, it is shown that $\rho \prec \eta$ if and only if $\spec(\rho(t)) \subset \spec(\eta(t))$, for every $t \in \C[G]$.

\begin{proof}[Proof of Theorem~\ref{thm:spectrum_cayley_d2}]
	Theorem~\ref{thm:spectrum_schreier} shows that $\spec(M_\xi) = \left\lbrack -\frac{1}{2^{m-1}}, 0 \right\rbrack \cup \left\lbrack 1 -\frac{1}{2^{m-1}}, 1 \right\rbrack$, for any $\xi \in X^\N$. Since $G$ is amenable, we know by Hulanicki's Theorem that $\lambda_{G/\Stab(\xi)} \prec \lambda_G$, where $\lambda_{G/\Stab(\xi)}$ is the quasi-regular representation of $G$ in $\ell^2(G/\Stab(\xi))$, for any $\xi \in X^\N$. Considering $M = \frac{1}{|S|}\sum_{s \in S}s \in \C[G]$, this implies that $\spec(\lambda_{G/\Stab(\xi)}(M)) \subset \spec(\lambda_G(M))$, or, equivalently, $\spec(M_\xi) \subset \spec(G)$, for any $\xi \in X^\N$.
	
	To prove the other inclusion, consider the element $t \in \C[G]$, defined as follows:
	\[
	t = \frac{2}{|B|} \sum\limits_{b\in B}b - 1.
	\]	
	We observe that $t^2 = 1$. Indeed,
	\[
	t^2 = \left(\frac{2}{|B|} \sum\limits_{b\in B}b - 1\right)^2 = \frac{4}{|B|^2}\sum\limits_{b\in B}\sum\limits_{b'\in B}bb' + 1 - \frac{4}{|B|}\sum\limits_{b\in B}b = \]\[ = \frac{4}{|B|^2}\sum\limits_{b\in B}\sum\limits_{c\in B}c + 1 - \frac{4}{|B|}\sum\limits_{b\in B}b = \frac{4}{|B|}\sum\limits_{c\in B}c + 1 - \frac{4}{|B|}\sum\limits_{b\in B}b = 1.
	\]
	
	It follows that the subgroup $D = \langle a, t \rangle$ of the group algebra $\C[G]$ is a dihedral group (in fact infinite), as $a^2 = t^2 = 1$.
	
	Let $\rho = \lambda_G|_{D}$ be the restriction of the regular representation to $D \subset \C[G]$. Since both $a$ and $t$ are involutions, $\rho(a)$ and $\rho(t)$ are unitary operators, and hence $\rho$ is a unitary representation. By Hulanicki's Theorem, provided that $D$ is amenable, we have that $\rho \prec \lambda_D$, where $\lambda_D$ is the regular representation of $D$ in $\ell^2(D)$. This implies that $\spec(\rho(m)) \subset \spec(\lambda_D(m))$ for every $m \in \C[G]$. Notice that $M = \frac{a}{2^m} + \frac{t}{2} + \frac{2^{m-1}-1}{2^m} \in \C[D]$, so we have $\spec(G) = \spec(\lambda_G(M)) = \spec(\rho(M)) \subset \spec(\lambda_D(M))$.
	
	We only have to compute the latter, which is not hard to do as it corresponds to the spectrum of the Markov operator associated to a random walk on $\Z$ with $2$-periodic probabilities. In particular, the probability of staying at a vertex is $\frac{2^{m-1}-1}{2^m}$, and the probabilities of moving to a neighbor are $2$-periodic of values $\frac{1}{2}$ and $\frac{1}{2^m}$.
	
	To find the spectrum of a $2$-periodic graph we can use the elements of Floquet-Bloch theory (see for instance~\cite{BK13}). Let $k \in [-\pi, \pi]$ be a frequency and $e^{ikn}$ be its corresponding wave function. Using the $2$-periodicity of the graph we build a $2\times 2$ matrix for each $k$, we find its eigenvalues and we take the closure of their union for all $k$. The computations are shown below, and lead to the relation $\left\lbrack -\frac{1}{2^{m-1}}, 0 \right\rbrack \cup \left\lbrack 1 -\frac{1}{2^{m-1}}, 1 \right\rbrack = \spec(M_\xi)$.
	
	\[
	0 = \left|\begin{matrix}
	\frac{2^{m-1}-1}{2^m} - x & \frac{1}{2^m} + \frac{1}{2}e^{-ik} \\[6pt]
	\frac{1}{2^m} + \frac{1}{2}e^{ik} & \frac{2^{m-1}-1}{2^m} - x
	\end{matrix}\right| = \left(\frac{2^{m-1}-1}{2^m} - x\right)^2 - \left( \frac{1}{4^m} + \frac{1}{4} + \frac{1}{2^m}\cos(k) \right).
	\]\[
	x = \frac{2^{m-1}-1}{2^m} \pm \frac{1}{2^m}\sqrt{4^{m-1} + 1 + 2^m\cos(k)}.
	\]\[
	\spec(\lambda_D(M)) = \bigcup_{k\in [-\pi, \pi]} \spec\left(\lambda_D(M)_k\right) = \]\[ = \frac{2^{m-1}-1}{2^m} \pm \frac{1}{2^m}[2^{m-1} - 1, 2^{m-1} + 1] = \left\lbrack -\frac{1}{2^{m-1}}, 0 \right\rbrack \cup \left\lbrack 1 -\frac{1}{2^{m-1}}, 1 \right\rbrack.
	\]
\end{proof}

We can therefore conclude in Corollary~\ref{cor:regular_spectrum_d2} that, for spinal groups, as for many other classes of groups, the spectrum of the Cayley graph does not determine the group.

\regularspectrumdtwo*
\begin{proof}
	Theorem~\ref{thm:spectrum_cayley_d2} shows that $\spec(G)$ for spinal groups acting on the binary tree depends only on $m$, one of the parameters in the definition of spinal groups. Hence, all spinal groups with $d = 2$ and a fixed $m \ge 1$ will share the same spectrum. For $m=2$, we obtain the family of groups defined by Grigorchuk in~\cite{Gri84}. This family contains uncountably many groups with different growth function, which is a quasi-isometric invariant. Hence, there are uncountably many isospectral groups which are pairwise non quasi-isometric.
\end{proof}

\section{Dependence of the spectrum on the generating set}
\label{sec:dependence_gen_set}

All the results discussed so far concerned spinal groups with the spinal generators $S = (A \cup B)\setminus\{1\}$. One might also wonder what are the spectra like if we consider different generating sets, for instance minimal ones.

For spinal groups acting on the binary tree ($d=2$), the infinite Schreier graphs $\Gamma_\xi$ have linear shape. The Schreier graphs of a minimal generating set can then be obtained by erasing double edges in the Schreier graph $\Gamma_\xi$ corresponding to the spinal generators $S$. This can be translated into considering a Markov operator on $\Gamma_\xi$ with non-uniform distribution of probabilities on $S$. The spectra of such anisotropic Markov operators were studied by Lenz, Sell and two first-named authors in~\cite{GLNS19}. Their results imply the following.



\begin{prop}
\label{prop:type_spectrum}
Let $G_\omega$ be a spinal group with $d=2$, $m \ge 2$ and $\omega \in \Omega_{2, m}$, and let $M_\xi^T$ be the Markov operator on the graph $\Gamma_\xi^T$, with generating set $T \subset S$. For $\pi \in \Epi(B, A)$, we define $q_\pi = |T \cap B \setminus \Ker(\pi)|$. Two cases may occur:
\begin{itemize}
\item If the numbers $q_\pi$ are not all equal over $\pi \in \Epi(B, A)$ appearing infinitely often in $\omega$, then $\spec(M_\xi^T)$ is a Cantor set of Lebesgue measure zero.
\item If not, $\spec(M_\xi^T)$ is a union of intervals.
\end{itemize}
\end{prop}
\begin{proof}
First notice that $a \in T$, or else $T$ would generate a finite group. We observe that we can relabel the vertices in $\Gamma_\xi^T$ by $\Z$ in such a way that the number of edges between them is the following:
\begin{itemize}
\item There is one $a$-edge between any vertex $v \in 2\Z$ and $v+1$.
\item There are $|T \cap B \setminus \Ker(\omega_0)| = q_{\omega_0}$ edges, between any vertex $v \in 4\Z + 1$ and $v+1$, and $|T \cap \Ker(\omega_0)|$ loops on each of $v$, $v+1$.
\item There are $|T \cap B \setminus \Ker(\omega_1)| = q_{\omega_1}$ edges between any vertex $v \in 8\Z + 3$ and $v+1$, and $|T \cap \Ker(\omega_1)|$ loops on each of $v$, $v+1$.

...

\item In general, for every $i \ge 0$, there are $|T \cap B \setminus \Ker(\omega_i)|$ edges between any vertex $v \in 2^{i+2}\Z + 2^{i+1} - 1$ and $v+1$, and $|T \cap \Ker(\omega_i)|$ loops on each of $v$, $v+1$.
\end{itemize}

Hence, the simple random walk on $\Gamma_\xi^T$ is given by a weighted random walk on $\Z$, defined by the following probabilities:
\begin{itemize}
\item Probability of $\frac{1}{|T|}$ of transitioning between any vertex $v \in 2\Z$ and $v+1$.
\item For every $i \ge 0$, probability of $\frac{q_{\omega_i}}{|T|}$ of transitioning between any vertex $v \in 2^{i+2}\Z + 2^{i+1} - 1$ and $v+1$.
\item For every $i \ge 0$, probability of $1 - \frac{q_{\omega_i}}{|T|} - \frac{1}{|T|}$ of staying at any vertex $v \in 2^{i+2}\Z + 2^{i+1} - 1$ or $v+1$.
\end{itemize}

These probabilities follow a periodic pattern if and only if the numbers $q_\pi$ are all equal for every $\pi \in \Epi(B, A)$ occurring infinitely often in $\omega$. If this is not the case, we may use Corollary 7.2 in~\cite{GLNS19} to obtain that $\spec(M_\xi^T)$ is a Cantor set of Lebesgue measure zero.

Suppose now that $q_\pi$ are all equal for every $\pi \in \Epi(B, A)$ occurring infinitely often in $\omega$, so the probabilities are periodic on $\Z$, let $l$ be that period. We may compute $\spec(M_\xi^T)$ with Floquet-Bloch theory. To do so, we first compute the spectrum of a fundamental domain, parametrized by $k\in[-\pi, \pi]$, with some boundary conditions. This gives a set of eigenvalues $\{x_1(k), \dots, x_l(k)\}$, which are the roots of a polynomial of degree $l$. These polynomials only depend on $\cos(k)$. Now $\spec(M_\xi^T)$ is just the union of these sets of eigenvalues for all $k\in[-\pi, \pi]$. Since the roots will vary continuously as $\cos(k) \in [-1, 1]$, this union will be a union of at least one and at most $l$ intervals.
\end{proof}

We already know from Theorem~\ref{thm:spectrum_schreier} that 
the second option in Proposition~\ref{prop:type_spectrum} is realized when $T=S$, the spinal generating set. The following result states that, except one degenerate example corresponding to $G_\omega = D_\infty$ ($d=2$, $m=1$), every spinal group on the binary tree has a generating set which gives a Cantor spectrum.

\generatingsetcantor*
\begin{proof}
Let $\pi, \pi' \in Epi(B, A)$ be two different epimorphisms occurring infinitely often in $\omega$. Recall that $B$ is a vector space over $\Z/2\Z$, and let $K = \Ker(\pi)$ and $K' = \Ker(\pi')$. We know that $[B:K] = [B:K'] = 2$, and $[K:K \cap K'] = 2$ because $\pi'$ surjects $K$ onto $A$ with kernel $K \cap K'$, since $K \neq K'$. Hence, we have $[B:K \cap K'] = 4$. In particular, we can choose $m-2$ elements $x_1, \dots, x_{m-2} \in K \cap K'$ which generate $K \cap K'$. Moreover, we can choose elements $y \in K \setminus K'$ and $y' \in K' \setminus K$ to complete the generating set to one of $K$ or $K'$, respectively, and such that $\{x_1, \dots, x_{m-2}, y, y'\}$ generate $B$.

If we now define $T = \{a, x_1, x_2, \dots, x_{m-2}, y, yy'\} \subset S$, it is clear that it is a minimal generating set for $G_\omega$, since $|T| = m+1$. Moreover, we have $q_\pi = |T \cap B\setminus K| = |\{yy'\}| = 1$ and $q_{\pi'} = |T \cap B \setminus K'| = |\{y, yy'\}| = 2$. By Proposition~\ref{prop:type_spectrum}, $\spec(M_\xi^T)$ is a Cantor set of Lebesgue measure zero.
\end{proof}

One can also find a condition on the generating set $T \subset S$ under which the spectrum on the Schreier graphs is one interval, for certain $G_\omega$.

\begin{prop}
Let $G_\omega$ be a spinal group with $d=2$, $m \ge 2$ and $\omega \in \Omega_{d,m}$, with generating set $T \subset S$. If $q_{\omega_i} = 1$ for every $i \ge 0$, then $\spec(M_\xi^T)$ is the interval $[1- \frac{4}{|T|}, 1]$.
\end{prop}
\begin{proof}
In the proof of Proposition~\ref{prop:type_spectrum}, we established that the simple random walk on $\Gamma_\xi^T$ is given by a weighted random walk on $\Z$. Since $q_{\omega_i} = 1$ for every $i \ge 0$, the probabilities can be simplified to:
\begin{itemize}
\item Probability $\frac{1}{|T|}$ of transitioning between any vertex $v \in \Z$ and $v+1$.
\item Probability $1 - \frac{2}{|T|}$ of staying at any vertex $v \in \Z$.
\end{itemize}
The fact that the probabilities are periodic allows us to use Floquet-Bloch theory in order to find $\spec(M_\xi^T)$, and since the period is $1$, the computation is rather simple. The only eigenvalue of a fundamental domain, parametrized by $k \in [-\pi, \pi]$, is the solution of the equation
\[
0 = 1 - \frac{2}{|T|} + \frac{1}{|T|}e^{ik} + \frac{1}{|T|}e^{-ik} - x  = 1 - \frac{2}{|T|} + \frac{2}{|T|}\cos(k) - x \Longrightarrow x = 1 - \frac{2}{|T|} + \frac{2}{|T|}\cos(k).
\]
Finally,
\[
\spec(M_\xi^T) = \bigcup_{k\in [-\pi, \pi]} \spec(M_\xi^T(k)) = \bigcup_{k\in [-\pi, \pi]} \left\{ 1 - \frac{2}{|T|} + \frac{2}{|T|}\cos(k) \right\} = \left[1-\frac{4}{|T|}, 1\right].
\]
\end{proof}

Self-similar groups inside the family of spinal groups were studied by \v{S}uni\'c~\cite{sunic}. For every $d$ and $m$, there are finitely many of them. They can be specified in terms of an epimorphism $\alpha \in \Epi(B, A)$ and an automorphism $\rho \in \Aut(B)$. The groups in \v{S}uni\'c's family are then the spinal groups defined by the periodic sequence $\omega = (\omega_n)_n$ given by $\omega_n = \alpha \rho^n$. Moreover, it was shown that any of these groups admits a natural minimal \v{S}uni\'c generating set $T = \{a, b_1, \dots, b_m\}$, contained in the spinal generating set $S$, such that
\[
a = (1,1)\sigma \quad b_1 = (1, b_2) \quad b_2 = (1, b_3) \quad \dots \quad b_{m-1} = (1, b_m) \quad b_m = (a, b'),
\]
for some $b' \in B$. Notice that, for $i = 1, \dots, m-1$, $\alpha(b_i) = 1$ and $\rho(b_i) = b_{i+1}$, while $\alpha(b_m) = a$ and $\rho(b_m) = b'$. The choice of this $b' \in B$ in such a way that $\rho$ is an automorphism will then determine the group. It was shown in~\cite{sunic} that a \v{S}uni\'c group is infinite torsion if and only if all $\rho$-orbits intersect $\Ker(\alpha)$.

\begin{ex}[Grigorchuk's group]
	Grigorchuk's group is the group $G$ in \v{S}uni\'c's family with $d=2$, $m=2$, $A = \{1, a\}$, $B = \{1, b_1, b_2, b_1b_2\}$ and $\rho(b_2) = b_1b_2$. With the standard notation $b, c, d$ for the generators, we have $b_1 = d$, $b_2 = b$ and $b_1b_2 = c$. The only nontrivial $\rho$-orbit is $b_1 \mapsto b_2 \mapsto b_1b_2 \mapsto b_1$, which intersects $\Ker(\alpha)$ at $b_1$, hence the group is infinite torsion. The minimal \v{S}uni\'c generating set is $T = \{a, b_1, b_2\}$ and the spinal generating set is $S = \{a, b_1, b_2, b_1b_2\}$, with
	\[
	a = (1,1)\sigma \quad b_1 = (1, b_2) \quad b_2 = (a,b_1b_2) \quad b_1b_2 = (a, b_1).
	\]
	We have
	\[
	\spec(M_\xi^S) = \spec(G, S) = \left[-\frac{1}{2}, 0\right]\cup\left[\frac{1}{2}, 1\right].
	\]
	We may consider any of the minimal generating sets $T_{b_1} = \{a, b_2, b_1b_2\}$, $T_{b_2} = \{a, b_1, b_1b_2\}$ or $T = \{a, b_1, b_2\}$. In that case, all of $\spec(M_\xi^{T_{b_1}})$, $\spec(M_\xi^{T_{b_2}})$ and $\spec(M_\xi^{T})$ are Cantor sets, for any $\xi \in X^\N$. It would be interesting to know, for these minimal generating sets, what is the spectrum on the Cayley graph. So far, we only know it must contain this Cantor set.
\end{ex}

\begin{ex}
One natural choice in the construction of \v{S}uni\'c groups above is to take $d=2$ and $\rho$ such that $b' = b_1$. This gives one group for each $m \ge 2$, we call them $G_m$. We have
\[
a = (1,1)\sigma \quad b_1 = (1, b_2) \quad b_2 = (1, b_3) \quad \dots \quad b_{m-1} = (1, b_m) \quad b_m = (a, b_1).
\]
The element $a b_1 \dots b_m$ is of infinite order. We consider two generating sets: the spinal generating set $S = (A\cup B)\setminus\{1\}$, of size $2^m$, and the \v{S}uni\'c minimal generating set $T = \{a, b_1, \dots, b_m\}$, of size $m+1$. On the one hand, Theorem~\ref{thm:spectrum_cayley_d2} yields that, for any $\xi \in X^\N$,
\[
\spec(M_\xi^S) = \spec(G_m, S) = \left[-\frac{1}{2^{m-1}}, 0\right] \cup \left[1 - \frac{1}{2^{m-1}}, 1\right].
\]
On the other hand,
\[
\spec(M_\xi^T) = \left[\frac{m-3}{m+1}, 1\right].
\]
Indeed, for any two-ended $\Gamma_\xi^T$, the simple random walk translates into the weighted random walk on $\Z$ with probability $\frac{1}{m+1}$ of moving to a neighbor and probability $\frac{m-1}{m+1}$ of staying on any vertex. By taking a one-vertex fundamental domain parametrized by $k \in [-\pi, \pi]$ and using Floquet-Bloch theory, we have:
\[
0 = \frac{m-1}{m+1} + \frac{1}{m+1}(e^{ik} + e^{-ik}) - x \Longrightarrow x = \frac{1}{m+1}(m-1 + 2\cos(k))
\]
\[
\spec\left(M_\xi^T\right) = \bigcup_{k\in [-\pi, \pi]} \spec\left(M_\xi^T(k)\right) = \bigcup_{k\in [-\pi, \pi]} \left\{ \frac{1}{m+1}(m-1 + 2\cos(k)) \right\} = \left[\frac{m-3}{m+1}, 1\right].
\]

\begin{prop}
\label{prop:sunic-G-m}

Let $G$ be a \v{S}uni\'c group with $d=2$ and $m \ge 2$, with minimal \v{S}uni\'c generating set $T$. Then the spectrum on the Schreier graph with respect to $T$ is $\left[\frac{m-3}{m+1}, 1\right]$ if $G = G_m$ or a Cantor set of zero Lebesgue measure otherwise.
\end{prop}
\begin{proof}
We found above the spectrum on the Schreier graphs for the groups $G_m$. Suppose now that $\spec(M_\xi^T)$ is a union of intervals. By Proposition~\ref{prop:type_spectrum}, we have that the numbers $q_\pi$ are all equal over $\pi \in \Epi(B, A)$ occurring infinitely often in $\omega$. By definition of the minimal \v{S}uni\'c generating set $T$, we know that $q_{\omega_0} = m-1$, so, as $\omega$ is periodic, $q_{\omega_n} = m-1$ for every $n \ge 0$.

Now, for any $k = 0, \dots, m-1$, we know that $\omega_k(b_{m-k}) = \omega_0\rho^{k}(b_{m-k}) = \omega_0(b_m) = a$, so $b_{m-k} \not\in \Ker(\omega_k)$. As $q_{\omega_k} = m-1$, the only possibility is that, for every $j = 0, \dots, m-1$, $b_j \in \Ker(\omega_k)$ if and only if $j \neq k$. In particular, this implies that $b_m = (a, b_1)$, so that we are in fact in the case of the group $G_m$.
\end{proof}

We also have

\begin{lem}
The Cayley graph of $G_m$ with generating set $T$ is bipartite, for any $m \ge 2$.
\end{lem}
\begin{proof}
We only have to show that all relations in the group $G_m$ have even length. Let $w$ be a freely reduced word on $T$, and let $|w|$ represent its length and $|w|_t$ the number of times the generator $t \in T$ occurs in $w$.

Suppose that $w$ represents the identity element of $G_m$. In that case, $|w|_a$ must be even, or otherwise its action on the first level would be nontrivial. This allows us to write the word $w$ as a product of $b_i$ and $b_i^a$. Let $w_0$ and $w_1$ be the two projections of the word $w$ into the first level, before reduction. Let us look at the decomposition of a generator $t \in T$. If $t = a$, then it decomposes as $1$ on both subtrees. If $t = b_i$, then it decomposes as $b_{i+1}$ on the right and as $1$ on the left, or as $a$ if $i = m$. Notice that the decomposition of $b_i^a$ is that of $b_i$ exchanging the two projections.

It is clear that both $w_0$ and $w_1$ represent the identity, too. Hence, $|w_0|_a$ and $|w_1|_a$ must both be even as well. But $|w_0|_a + |w_1|_a = |w|_{b_m}$, so $|w|_{b_m}$ must also be even.

By iterating this argument we can conclude that $|w|$ must be even. In general, let $w_u$ be the projection of $w$ onto the vertex $u$ in $X^k$, the $k$-th level of the tree, for $1 \ge k \ge m$. For any $u \in X^k$, $w_u$ must represent the identity, and hence $|w_u|_a$ must be even. But tracing back the $a$'s occurring in $w_u$ we obtain
\[
	\sum_{u \in L_k} |w_u|_a = \sum_{u \in L_{k-1}} |w_u|_{b_m} = \dots = \sum_{u \in L_1} |w_u|_{b_{m-k+2}} = |w|_{b_{m-k+1}}.
\]
This shows that $|w|_t$ is even, for every $t \in T$, which implies that $|w|$ is indeed even and hence that the Cayley graph with the generating set $T$ is bipartite.
\end{proof}

The spectrum of a bipartite graph is symmetric about $0$. At the same time, for amenable groups, the spectrum on any Schreier graph is contained in the spectrum on the Cayley graph. Hence we have,
\[
\spec(G_m^T) \supset -\spec(M_\xi^T) \cup \spec(M_\xi^T) = \left[-1, \frac{3-m}{m+1}\right] \cup \left[\frac{m-3}{m+1}, 1\right].
\]

Two cases are of special interest: $m = 2$ and $m = 3$. In these cases, the union of the two intervals above is the whole interval $[-1, 1]$ and we can thus conclude that the spectrum of the Cayley graphs of $G_2$ and $G_3$ with respect to the minimal \v{S}uni\'c generating set is the whole interval $[-1, 1]$. For $m \ge 4$, the union of intervals is actually disjoint, so we can only conclude that the spectrum of the Cayley graph contains two intervals $[-1, -\beta]$ and $[\beta, 1]$, with $\beta > 0$.

Note that the group $G_2$ was studied in~\cite{Ers04} and is therefore sometimes called the Grigorchuk-Erschler group. It is the only self-similar group in the Grigorchuk family (spinal groups with $d=2$ and $m=2$) besides Grigorchuk's group. The group $G_3$ is known as Grigorchuk's overgroup~\cite{BG00} because it contains Grigorchuk's group as a subgroup. Indeed, the automorphisms $b_2b_3$, $b_1b_3$, $b_1b_2$ are the generators $b$, $c$, $d$ of Grigorchuk's group.

\spectrumgego*

\end{ex}

\bibliography{my.bib}{}
\bibliographystyle{plain}

\end{document}

%% file: img/grigorchuk-3.tex
\begin{center}
	\begin{figure}[h]
		\begin{tikzpicture}
		\draw[color=blue] (0,0) .. controls (0.5, 1) and (-0.5, 1) .. (0,0);
		\draw[color=red, rotate = 90] (0,0) .. controls (0.5, 1) and (-0.5, 1) .. (0,0);
		\draw[color=green, rotate = 180] (0,0) .. controls (0.5, 1) and (-0.5, 1) .. (0,0);
		\draw (-0.9*1, 0) node {\tiny{$c$}};
		\draw (0, 0.9*1) node {\tiny{$b$}};
		\draw (0, -0.9*1) node {\tiny{$d$}};
		
		
		\draw (0,0) -- (1,0);
		\draw (0.5,0) node[above] {\tiny{$a$}};
		
		
		\draw[color=green] (1,0) .. controls (1.5, 1) and (0.5, 1) .. (1,0);
		\draw[color=green] (2,0) .. controls (2.5, 1) and (1.5, 1) .. (2,0);
		\draw[color=blue] (1,0) .. controls (1.2, 0.2) and (1.8, 0.2) .. (2,0);
		\draw[color=red] (1,0) .. controls (1.2, -0.2) and (1.8, -0.2) .. (2,0);
		
		\draw (1.5, 0.1) node[above] {\tiny{$b$}};
		\draw (1.5, -0.1) node[below] {\tiny{$c$}};
		\draw (1, 0.9) node {\tiny{$d$}};
		\draw (2, 0.9) node {\tiny{$d$}};
		
		
		\draw (2,0) -- (3,0);
		\draw (2.5,0) node[above] {\tiny{$a$}};
		
		
		\draw[color=red] (3,0) .. controls (3.5, 1) and (2.5, 1) .. (3,0);
		\draw[color=red] (4,0) .. controls (4.5, 1) and (3.5, 1) .. (4,0);
		\draw[color=blue] (3,0) .. controls (3.2, 0.2) and (3.8, 0.2) .. (4,0);
		\draw[color=green] (3,0) .. controls (3.2, -0.2) and (3.8, -0.2) .. (4,0);
		
		\draw (3.5, 0.1) node[above] {\tiny{$b$}};
		\draw (3.5, -0.1) node[below] {\tiny{$d$}};
		\draw (3, 0.9) node {\tiny{$c$}};
		\draw (4, 0.9) node {\tiny{$c$}};
		
		
		\draw (4,0) -- (5,0);
		\draw (4.5,0) node[above] {\tiny{$a$}};
		
		
		\draw[color=green] (5,0) .. controls (5.5, 1) and (4.5, 1) .. (5,0);
		\draw[color=green] (6,0) .. controls (6.5, 1) and (5.5, 1) .. (6,0);
		\draw[color=blue] (5,0) .. controls (5.2, 0.2) and (5.8, 0.2) .. (6,0);
		\draw[color=red] (5,0) .. controls (5.2, -0.2) and (5.8, -0.2) .. (6,0);
		
		\draw (5.5, 0.1) node[above] {\tiny{$b$}};
		\draw (5.5, -0.1) node[below] {\tiny{$c$}};
		\draw (5, 0.9) node {\tiny{$d$}};
		\draw (6, 0.9) node {\tiny{$d$}};
		
		
		\draw (6,0) -- (7,0);
		\draw (6.5,0) node[above] {\tiny{$a$}};
		
		
		\draw[color=blue] (7,0) .. controls (7.5, 1) and (6.5, 1) .. (7,0);
		\draw[shift={(7,0)}, rotate = -90, color=red] (0,0) .. controls (0.5, 1) and (-0.5, 1) .. (0,0);
		\draw[shift={(7,0)}, rotate = 180, color=green] (0,0) .. controls (0.5, 1) and (-0.5, 1) .. (0,0);
		\draw (7.9, 0) node {\tiny{$c$}};
		\draw (7, 0.9) node {\tiny{$b$}};
		\draw (7, -0.9) node {\tiny{$d$}};
		
		
		\draw[fill=black] (0,0) circle (1/32);
		\draw[fill=black] (1,0) circle (1/32);
		\draw[fill=black] (2,0) circle (1/32);
		\draw[fill=black] (3,0) circle (1/32);
		\draw[fill=black] (4,0) circle (1/32);
		\draw[fill=black] (5,0) circle (1/32);
		\draw[fill=black] (6,0) circle (1/32);
		\draw[fill=black] (7,0) circle (1/32);
		
		\draw (0,0) node[below] {\tiny{$110$}};
		\draw (1,0) node[below] {\tiny{$010$}};
		\draw (2*1,0) node[below] {\tiny{$000$}};
		\draw (3*1,0) node[below] {\tiny{$100$}};
		\draw (4*1,0) node[below] {\tiny{$101$}};
		\draw (5*1,0) node[below] {\tiny{$001$}};
		\draw (6*1,0) node[below] {\tiny{$011$}};
		\draw (7*1,0) node[below] {\tiny{$111$}};

		\end{tikzpicture}
		
		\caption{The graph $\Gamma_3$ for Grigorchuk's group (see Section~\ref{sec:preliminaries}).}
		\label{fig:grigorchuk-3}
	\end{figure}
\end{center}

%% file: img/gupta-fabrykowski-3.tex
\begin{center}
	\begin{figure}[h]
		\begin{tikzpicture}
		\tikzset{midarrow>/.style={decoration={markings, mark=at position 0.5 with {\arrow{>}}}, postaction={decorate}}}
		
		\tikzset{
			pics/Loop/.style n args={4}{
				code = {
					\begin{scope}[shift={#1}, rotate=#2]
					\draw[color=#4] (0,0) .. controls (#3/2, #3) and (-#3/2, #3) .. (0,0);
					\draw[fill=black] (0,0) circle (#3/32);
					\end{scope}
				}	
			}
		}
		
		\tikzset{
			pics/Line/.style n args={4}{
				code = {
					\begin{scope}[shift={#1}, rotate=#2]
					\draw[color=#4] (0,0) -- (0: #3);
					\end{scope}
				}	
			}
		}
		
		\tikzset{
			pics/Triangle/.style n args={4}{
				code = {
					\begin{scope}[shift={#1}, rotate=#2]
					\draw[color=#4] (0,0) -- (0: #3) -- (60: #3) -- cycle;
					\draw[fill=black] (0,0) circle (#3/32);
					\draw[fill=black] (0: #3) circle (#3/32);
					\draw[fill=black] (60: #3) circle (#3/32);
					\end{scope}
				}	
			}
		}

		\tikzset{
			pics/Gamma/.style ={
				code = {
					\draw (0,0) .. controls (\r, -\r) and (2*\r, -\r) .. (3*\r,0);
					\draw (0,0) .. controls (\r, \r) and (2*\r, \r) .. (3*\r,0);
					\draw (1.5*\r,0) node {\tiny $\Gamma_{#1}$};
				}
			}
		}

		\pgfmathsetmacro{\r}{1}
		
		\tikzset{
			pics/Xi/.style n args={3}{
				code = {
					\begin{scope}[shift={#1}]
					\pic {Loop={(0,0)}{60+#2}{0.5*\r}{blue}};
					\pic {Loop={(0,0)}{-60+#2}{0.5*\r}{red}};
					\draw (120 + #2: 0.35*\r) node {\tiny{$b$}};
					\draw (60 + #2: 0.35*\r) node {\tiny{$b^2$}};
					\draw[fill=black] (0,0) circle (\r/32);
					\draw (0+#2: 0.2*\r) node {\tiny{$#3$}};
					\end{scope}
				}	
			}
		}
		
		\tikzset{
			pics/Theta/.style n args={2}{
				code = {
					\begin{scope}[shift={#1}, rotate=#2]		  			
					\draw (0:0) edge[midarrow>] (0: \r);
					\draw (0:\r) edge[midarrow>] (60: \r);
					\draw (60:\r) edge[midarrow>] (0:0);
					
					\draw (0:0) edge[midarrow>, bend left] (60: \r);
					\draw (60:\r) edge[midarrow>, bend left] (0: \r);
					\draw (0:\r) edge[midarrow>, bend left] (0:0);
					
					\end{scope}
				}	
			}
		}
		
		\tikzset{
			pics/Lambda/.style n args={2}{
				code = {
					\begin{scope}[shift={#1}, rotate=#2]
					\pic {Triangle={(0,0)}{0 + #2}{\r}{blue}};
					
					\draw[color=blue] (0:0) edge[midarrow>] (0: \r);
					\draw[color=blue] (0:\r) edge[midarrow>] (60: \r);
					\draw[color=blue] (60:\r) edge[midarrow>] (0:0);
					
					\draw (0.5*\r, 0.15*\r) node {\tiny{$b$}};
					\draw (0.65*\r, 0.4*\r) node {\tiny{$b$}};
					\draw (0.35*\r, 0.4*\r) node {\tiny{$b$}};
					
					\draw[color=red] (0:0) edge[bend left, midarrow>] (60: \r);
					\draw[color=red] (60:\r) edge[bend left, midarrow>] (0: \r);
					\draw[color=red] (0:\r) edge[bend left, midarrow>] (0:0);
					
					\draw (0.5*\r, -0.3*\r) node {\tiny{$b^2$}};
					\draw (\r, 0.6*\r) node {\tiny{$b^2$}};
					\draw (0, 0.6*\r) node {\tiny{$b^2$}};
					\end{scope}
				}	
			}
		}

		\tikzset{
			pics/Gamma2/.style n args={3}{
				code = {
					\begin{scope}[shift={#1}, rotate=#2]
					\pic {Lambda={(0,0)}{0 + #2}};
					
					\begin{scope}[shift={(0,0)}]
					\draw (0,-0.2*\r) node {\tiny{$02#3$}};
					\pic {Theta={(0+#2:0)}{180+#2}};
					\pic {Xi={(180+#2:\r)}{60+#2}{12#3}};
					\pic {Xi={(240+#2:\r)}{180+#2}{22#3}};
					\end{scope}
					
					\begin{scope}[shift={(60:\r)}]
					\draw (0,-0.2*\r) node {\tiny{$01#3$}};
					\pic {Theta={(0+#2:0)}{60+#2}};
					\pic {Xi={(60+#2:\r)}{300+#2}{11#3}};
					\pic {Xi={(120+#2:\r)}{60+#2}{21#3}};
					\end{scope}
					
					\begin{scope}[shift={(0:\r)}]
					\draw (0,-0.2*\r) node {\tiny{$00#3$}};
					\pic {Theta={(0+#2:0)}{300+#2}};
					\pic {Xi={(300+#2:\r)}{180+#2}{10#3}};
					\end{scope}
					\end{scope}
				}	
			}
		}
		
		\pic {Lambda={(0,0)}{0}};
		
		\begin{scope}[shift={(0 : 0)}]
		\draw (270 : 0.2*\r) node {\tiny{$201$}};
		\begin{scope}[shift={(180 : 2*\r)}]
		\pic {Gamma2={(0, 0)}{0}{1}};
		\end{scope}
		\end{scope}
		
		\begin{scope}[shift={(60 : \r)}]
		\draw (150 : 0.2*\r) node {\tiny{$200$}};
		\begin{scope}[shift={(60 : 2*\r)}]
		\pic {Gamma2={(0, 0)}{240}{0}};
		\end{scope}
		\end{scope}
		
		\begin{scope}[shift={(0 : \r)}]
		\draw (30 : 0.2*\r) node {\tiny{$202$}};
		\begin{scope}[shift={(300 : 2*\r)}]
		\pic {Gamma2={(0, 0)}{120}{2}};
		\end{scope}
		\end{scope}
		\end{tikzpicture}
		
		\caption{The graph $\Gamma_3$ for the Gupta-Fabrykowski group (see Section~\ref{sec:preliminaries}).}
		\label{fig:gupta-fabrykowski-3}
	\end{figure}
\end{center}

%% file: img/parabolas.tex
\begin{center}
	\begin{figure}[h]
		\Tiny
		\begin{tikzpicture}[scale=0.8]
		\begin{scope}[shift={(0,0)}]
		\draw (-4, -4) rectangle (4, 4);
		\draw (0, -4) parabola (-4, 4);
		\draw (0, -4) parabola (4, 4);
		
		\draw (-4, 4) node[above] {$-4$};
		\draw (4, 4) node[above] {$4$};
		\draw (-4, -4) node[left] {$-4$};
		\draw (-4, 4) node[left] {$4$};
		
		\draw[dotted] (4, 0) -- (-4, 0) node[left] {$0$};
		
		\draw[fill] (-2.82842712475,0) circle (0.05);
		\draw[fill] (2.82842712475,0) circle (0.05);
		
		\draw[dotted] (-2.82842712475, 0) -- (-2.82842712475, -4) node[below] {$F^{-1}(0)$};
		\draw[dotted] (2.82842712475, 0) -- (2.82842712475, -4) node[below] {$F^{-1}(0)$};
		
		\draw[dotted] (4, 2.82842712475) -- (-4, 2.82842712475) node[left] {$F^{-1}(0)$};
		\draw[dotted] (4, -2.82842712475) -- (-4, -2.82842712475) node[left] {$F^{-1}(0)$};
		
		\draw[dotted] (-3.69551813005, 2.82842712475) -- (-3.69551813005, -4.5) node[below] {$F^{-2}(0)$};
		\draw[dotted] (-1.53073372946, -2.82842712475) -- (-1.53073372946, -4.5) node[below] {$F^{-2}(0)$};
		
		\draw[dotted] (1.53073372946, -2.82842712475) -- (1.53073372946, -4.5) node[below] {$F^{-2}(0)$};
		\draw[dotted] (3.69551813005, 2.82842712475) -- (3.69551813005, -4.5) node[below] {$F^{-2}(0)$};
		
		\draw[fill] (-1.53073372946, -2.82842712475) circle (0.05);
		\draw[fill] (1.53073372946, -2.82842712475) circle (0.05);
		\draw[fill] (-3.69551813005, 2.82842712475) circle (0.05);
		\draw[fill] (3.69551813005, 2.82842712475) circle (0.05);
		\end{scope}
		
		
		\begin{scope}[scale = 4/3, shift={(8,0)}]
		\draw (-3, -3) rectangle (3, 3);
		\draw[smooth, domain=-3:-1.73205080757] plot (\x, {\x*\x - 6});
		\draw[smooth, domain=1.73205080757:3] plot (\x, {\x*\x - 6});
		
		\draw (-3, 3) node[above] {$-3$};
		\draw (3, 3) node[above] {$3$};
		\draw (-3, -3) node[left] {$-3$};
		\draw (-3, 3) node[left] {$3$};
		
		\draw[dotted] (3, 0) -- (-3, 0) node[left] {$0$};
		
		\draw[fill] (-2.44948974278, 0) circle (3*0.05/4);
		\draw[fill] (2.44948974278, 0) circle (3*0.05/4);
		
		\draw[dotted] (-2.44948974278, 0) -- (-2.44948974278, -3) node[below] {$F^{-1}(0)$};
		\draw[dotted] (2.44948974278, 0) -- (2.44948974278, -3) node[below] {$F^{-1}(0)$};
		
		\draw[dotted] (3, -2.44948974278) -- (-3, -2.44948974278) node[left] {$F^{-1}(0)$};
		\draw[dotted] (3, 2.44948974278) -- (-3, 2.44948974278) node[left] {$F^{-1}(0)$};
		
		\draw[fill] (-2.90680060252, 2.44948974278) circle (3*0.05/4);
		\draw[fill] (2.90680060252, 2.44948974278) circle (3*0.05/4);
		\draw[fill] (-1.88427977148, -2.44948974278) circle (3*0.05/4);
		\draw[fill] (1.88427977148, -2.44948974278) circle (3*0.05/4);
		
		\draw[dotted] (-2.90680060252, 2.44948974278) -- (-2.90680060252, -3.5) node[below] {$F^{-2}(0)$};
		\draw[dotted] (2.90680060252, 2.44948974278) -- (2.90680060252, -3.5) node[below] {$F^{-2}(0)$};
		\draw[dotted] (-1.88427977148, -2.44948974278) -- (-1.88427977148, -3.5) node[below] {$F^{-2}(0)$};
		\draw[dotted] (1.88427977148, -2.44948974278) -- (1.88427977148, -3.5) node[below] {$F^{-2}(0)$};
		\end{scope}
		\end{tikzpicture}
		
		\caption{Illustration of $F^{-k}(0)$ for $k = 0, 1, 2$, for the case $d=2$ (left, $F(x) = x^2 - 2$) and $d=3$ (right, $F(x) = x^2 - 6$). The set of all preimages is dense in the interval for $d=2$ but accumulates on a Cantor set for $d=3$.}
		\label{fig:parabolas}
	\end{figure}
\end{center}

%% file: img/spectra.tex
\begin{center}
	\begin{figure}[h]
		\begin{tikzpicture}
			\begin{scope}[scale=2, shift={(0,0)}]
			\draw (-1, 0) rectangle (1,1);
			
			\draw (-1, 0) node[below] {$-1$};
			\draw (-0.5, 0) node[below] {$-\frac{1}{2}$};
			\draw (0, 0) node[below] {$0$};
			\draw (0.5, 0) node[below] {$\frac{1}{2}$};
			\draw (1, 0) node[below] {$1$};
		
			\draw[line width=2] (-0.5, 0) -- (0, 0);
			\draw[line width=2] (0.5, 0) -- (1, 0);
			\end{scope}

			\begin{scope}[scale=2, shift={(3,0)}]
			\draw (-1, 0) rectangle (1,1);
			
			\draw (-1, 0) node[below] {$-1$};
			\draw (-0.5, 0) node[below] {$-\frac{1}{2}$};
			\draw (0, 0) node[below] {$0$};
			\draw (0.5, 0) node[below] {$\frac{1}{2}$};
			\draw (1, 0) node[below] {$1$};
			
			\pgfmathsetmacro{\mu}{1/3}
			\pgfmathsetmacro{\w}{6*\mu}
			\pgfmathsetmacro{\h}{2*\mu}
			
			\draw[line width=\w] (0.25, 0) -- (0.25, \h);
			
			\pgfmathsetmacro{\mu}{1/9}
			\pgfmathsetmacro{\w}{6*\mu}
			\pgfmathsetmacro{\h}{2*\mu}
			
			\draw[line width=\w] (0.86237243569, 0) -- (0.86237243569, \h);
			\draw[line width=\w] (-0.36237243569, 0) -- (-0.36237243569, \h);
			
			\pgfmathsetmacro{\mu}{1/27}
			\pgfmathsetmacro{\w}{6*\mu}
			\pgfmathsetmacro{\h}{2*\mu}
			
			\draw[line width=\w] (0.97670015062, 0) -- (0.97670015062, \h);
			\draw[line width=\w] (0.72106994287, 0) -- (0.72106994287, \h);
			\draw[line width=\w] (-0.47670015062, 0) -- (-0.47670015062, \h);
			\draw[line width=\w] (-0.22106994287, 0) -- (-0.22106994287, \h);
			
			\pgfmathsetmacro{\mu}{1/81}
			\pgfmathsetmacro{\h}{2*\mu}
			
			\draw[line width=\w] (0.99610658599, 0) -- (0.99610658599, \h);
			\draw[line width=\w] (0.95197399219, 0) -- (0.95197399219, \h);
			\draw[line width=\w] (0.75718094826, 0) -- (0.75718094826, \h);
			\draw[line width=\w] (0.68968734612, 0) -- (0.68968734612, \h);
			\draw[line width=\w] (-0.18968734612, 0) -- (-0.18968734612, \h);
			\draw[line width=\w] (-0.25718094826, 0) -- (-0.25718094826, \h);
			\draw[line width=\w] (-0.45197399219, 0) -- (-0.45197399219, \h);
			\draw[line width=\w] (-0.49610658599, 0) -- (-0.49610658599, \h);
			
			\pgfmathsetmacro{\mu}{1/243}
			\pgfmathsetmacro{\h}{2*\mu}
			
			\draw[line width=\w] (0.9993508167, 0) -- (0.9993508167, \h);
			\draw[line width=\w] (0.99195249042, 0) -- (0.99195249042, \h);
			\draw[line width=\w] (0.94636329349, 0) -- (0.94636329349, \h);
			\draw[line width=\w] (0.95837506807, 0) -- (0.95837506807, \h);
			\draw[line width=\w] (0.68413517883, 0) -- (0.68413517883, \h);
			\draw[line width=\w] (0.69666150713, 0) -- (0.69666150713, \h);
			\draw[line width=\w] (0.76485742052, 0) -- (0.76485742052, \h);
			\draw[line width=\w] (0.74820152843, 0) -- (0.74820152843, \h);
			\draw[line width=\w] (-0.4993508167, 0) -- (-0.4993508167, \h);
			\draw[line width=\w] (-0.49195249042, 0) -- (-0.49195249042, \h);
			\draw[line width=\w] (-0.44636329349, 0) -- (-0.44636329349, \h);
			\draw[line width=\w] (-0.45837506807, 0) -- (-0.45837506807, \h);
			\draw[line width=\w] (-0.18413517883, 0) -- (-0.18413517883, \h);
			\draw[line width=\w] (-0.19666150713, 0) -- (-0.19666150713, \h);
			\draw[line width=\w] (-0.26485742052, 0) -- (-0.26485742052, \h);
			\draw[line width=\w] (-0.24820152843, 0) -- (-0.24820152843, \h);
			
			\end{scope}
		\end{tikzpicture}

		\caption{$\spec(M_\xi)$ for Grigorchuk's group (left) and the Gupta-Fabrykowski group (right), for any $\xi \in X^\N$. The height represents the density of states $\nu$ of each point (in logarithmic scale).}
		\label{fig:spectra}
	\end{figure}
\end{center}

%% file: img/rhos.tex
\begin{center}
	\begin{figure}[h]
		\begin{tikzpicture}
		\begin{scope}[shift={(0,0)}]
			\draw (0,0) ellipse (0.5 and 1);
			\draw (0,0) node {$f$};
			\draw (0,2.5) node[below] {$\Gamma_n$};
	
			\draw (1.5,0) node {$\longmapsto$};
			\draw (1.5,0) node[above] {$\rho^i_n$};

			\begin{scope}[shift={(4.5,0)}]
				\draw (-1,0) ellipse (1 and 0.5);
				\draw (-1,0) node {$f$};
	
				\draw (0, 0) -- (30:0.75) -- (-30:0.75) -- cycle;
	
				\begin{scope}[shift={(30:0.75)}]
					\draw[rotate=60] (1,0) ellipse (1 and 0.5);
					\draw[rotate=60] (1,0) node {$0$};
				\end{scope}
	
				\begin{scope}[shift={(-30:0.75)}]
					\draw[rotate=-60] (1,0) ellipse (1 and 0.5);
					\draw[rotate=-60] (1,0) node {$0$};
				\end{scope}
	
				\draw (0,2.5) node[below] {$\Gamma_{n+1}$};
			\end{scope}
		\end{scope}

		\draw (7.5,-2) -- (7.5,2);

		\begin{scope}[shift={(9,0)}]
			\draw (0,0) ellipse (0.5 and 1);
			\draw (0,0) node {$f$};
			\draw (0,2.5) node[below] {$\Gamma_n$};
	
			\draw (1.5,0) node {$\longmapsto$};
			\draw (1.5,0) node[above] {$\rho_n$};

			\begin{scope}[shift={(4.5,0)}]
				\draw (-1,0) ellipse (1 and 0.5);
				\draw (-1,0) node {$f$};
	
				\draw (0, 0) -- (30:0.75) -- (-30:0.75) -- cycle;
	
				\begin{scope}[shift={(30:0.75)}]
					\draw[rotate=60] (1,0) ellipse (1 and 0.5);
					\draw[rotate=60] (1,0) node {$f$};
				\end{scope}
	
				\begin{scope}[shift={(-30:0.75)}]
					\draw[rotate=-60] (1,0) ellipse (1 and 0.5);
					\draw[rotate=-60] (1,0) node {$f$};
				\end{scope}
	
				\draw (0,2.5) node[below] {$\Gamma_{n+1}$};
			\end{scope}
		\end{scope}

		\end{tikzpicture}
		
		\caption{Sketch of the transition operators $\rho^i_n$ and $\rho_n$. The former copies $f$ on the $i$-th copy of $\Gamma_n$ in $\Gamma_{n+1}$ and vanishes elsewhere; the latter copies $f$ on all the copies of $\Gamma_n$.}
		\label{fig:rhos}
	\end{figure}
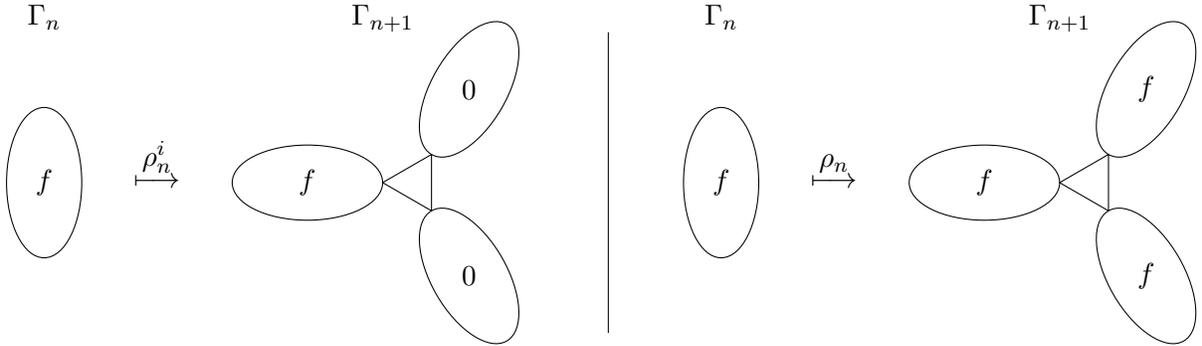
\end{center}

%% file: img/gupta-fabrykowski-fups-1.tex
\begin{center}
	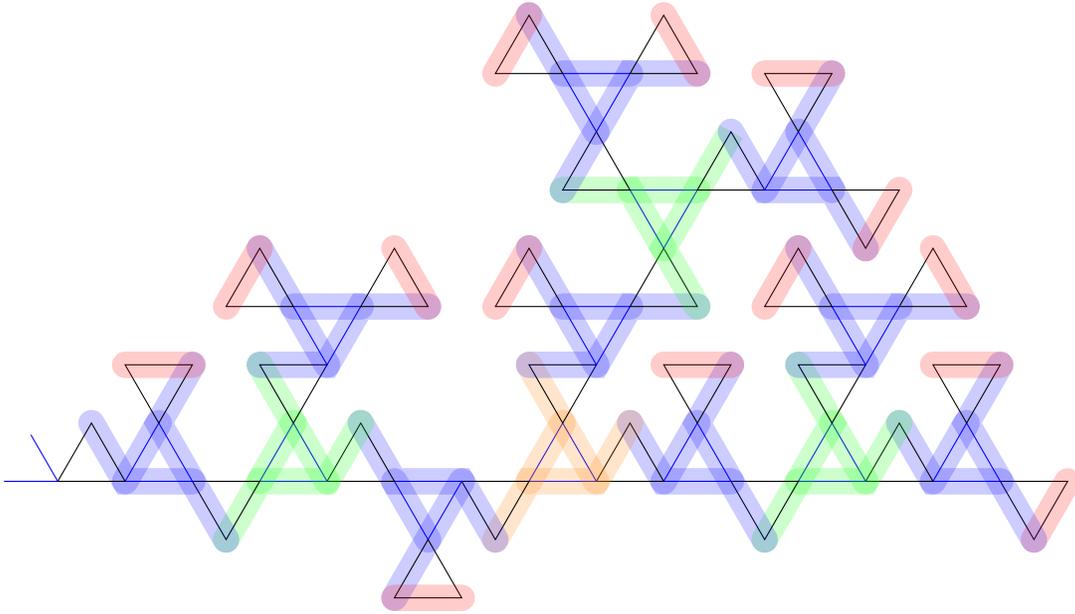
\begin{figure}[h]
		\begin{tikzpicture}[scale=0.895]
		\begin{scope}[shift={(0,0)}]
			\draw (0,0) -- (240:1) -- (-1, 0) -- cycle;
			\draw [line width=10, draw=red, opacity=0.2, line cap=round]
			(0,0) -- (240:1);
		\end{scope}
		
		\begin{scope}[shift={(-1,0)}]
			\draw[color=blue] (0,0) -- (120:1) -- (180:1) -- cycle;
			
			\begin{scope}[transparency group, opacity=0.2]
				\draw [line width=10, draw=blue, line cap=round]
				(300:1) -- (120:1) -- +(60:1);
				\draw [line width=10, draw=blue, line cap=round]
				(0,0) -- (180:1) -- +(120:1);
				\draw [line width=10, draw=blue, line cap=round]
				(120:1) -- (180:1);
			\end{scope}
			
			\begin{scope}[shift={(120:1)}]
				\draw (0,0) -- (60:1) -- (120:1) -- cycle;
				\draw [line width=10, draw=red, opacity=0.2, line cap=round]
				(60:1) -- (120:1);
			\end{scope}
			
			\begin{scope}[shift={(180:1)}]
				\draw (0,0) -- (120:1) -- (180:1) -- cycle;
			\end{scope}
		\end{scope}
		
		\begin{scope}[shift={(-3,0)}]
			\draw[color=blue] (0,0) -- (120:1) -- (180:1) -- cycle;
			
			\begin{scope}[transparency group, opacity=0.2]
				\draw [line width=10, draw=green, line cap=round]
				(120:2) -- (120:1) -- (180:1) -- +(240:1);
				\draw [line width=10, draw=green, line cap=round]
				(60:1) -- (0,0) -- (180:1);
				\draw [line width=10, draw=green, line cap=round]
				(0,0) -- (120:1);
			\end{scope}
			
			\begin{scope}[shift={(120:1)}]
				\draw (0,0) -- (60:1) -- (120:1) -- cycle;
				
				\begin{scope}[shift={(60:1)}]
					\draw[color=blue] (0,0) -- (60:1) -- (120:1) -- cycle;
					
					\begin{scope}[transparency group, opacity=0.2]
						\draw [line width=10, draw=blue, line cap=round]
						(0,0) -- (120:2);
						\draw [line width=10, draw=blue, line cap=round]
						(180:1) -- (0,0) -- (60:1) -- +(0:1);
						\draw [line width=10, draw=blue, line cap=round]
						(60:1) -- (120:1);
					\end{scope}
					
					\begin{scope}[shift={(60:1)}]
						\draw (0,0) -- (0:1) -- (60:1) -- cycle;
						\draw [line width=10, draw=red, opacity=0.2, line cap=round]
						(0:1) -- (60:1);
					\end{scope}
					
					\begin{scope}[shift={(120:1)}]
						\draw (0,0) -- (120:1) -- (180:1) -- cycle;
						\draw [line width=10, draw=red, opacity=0.2, line cap=round]
						(120:1) -- (180:1);
					\end{scope}
				\end{scope}
			\end{scope}
			
			\begin{scope}[shift={(180:1)}]
				\draw (0,0) -- (180:1) -- (240:1) -- cycle;
				
				\begin{scope}[shift={(180:1)}]
					\draw[color=blue] (0,0) -- (120:1) -- (180:1) -- cycle;
					
					\begin{scope}[transparency group, opacity=0.2]
						\draw [line width=10, draw=blue, line cap=round]
						(300:1) -- (120:1) -- +(60:1);
						\draw [line width=10, draw=blue, line cap=round]
						(0,0) -- (180:1) -- +(120:1);
						\draw [line width=10, draw=blue, line cap=round]
						(120:1) -- (180:1);
					\end{scope}
				
					\begin{scope}[shift={(120:1)}]
						\draw (0,0) -- (60:1) -- (120:1) -- cycle;
						\draw [line width=10, draw=red, opacity=0.2, line cap=round]
						(60:1) -- (120:1);
					\end{scope}
					
					\begin{scope}[shift={(180:1)}]
					\draw (0,0) -- (120:1) -- (180:1) -- cycle;
					\end{scope}
				\end{scope}
			\end{scope}
		\end{scope}
		
		\begin{scope}[shift={(-7,0)}]
			\draw[color=blue] (0,0) -- (120:1) -- (180:1) -- cycle;
			\begin{scope}[transparency group, opacity=0.2]
				\draw [line width=10, draw=orange, line cap=round]
				(60:1) -- (0,0) -- (180:1) -- +(240:1);
				\draw [line width=10, draw=orange, line cap=round] (0,0) -- (120:2);
				\draw [line width=10, draw=orange, line cap=round] (120:1) -- (180:1);
			\end{scope}
			
			\begin{scope}[shift={(120:1)}]
				\draw (0,0) -- (60:1) -- (120:1) -- cycle;
				
				\begin{scope}[shift={(60:1)}]
					\draw[color=blue] (0,0) -- (60:1) -- (120:1) -- cycle;
					
					\begin{scope}[transparency group, opacity=0.2]
						\draw [line width=10, draw=blue, line cap=round]
						(0,0) -- (120:2);
						\draw [line width=10, draw=blue, line cap=round]
						(180:1) -- (0,0) -- (60:1) -- +(0:1);
						\draw [line width=10, draw=blue, line cap=round]
						(60:1) -- (120:1);
					\end{scope}
					
					\begin{scope}[shift={(60:1)}]
						\draw (0,0) -- (0:1) -- (60:1) -- cycle;
						
						\begin{scope}[shift={(60:1)}]
							\draw[color=blue] (0,0) -- (60:1) -- (120:1) -- cycle;

							\begin{scope}[transparency group, opacity=0.2]
								\draw [line width=10, draw=green, line cap=round]
								(300:1) -- (0,0) -- (60:2) -- +(240:1);
								\draw [line width=10, draw=green, line cap=round] (0,0) -- (120:1) -- +(180:1);
								\draw [line width=10, draw=green, line cap=round] (60:1) -- (120:1);
							\end{scope}
							
							\begin{scope}[shift={(60:1)}]
								\draw (0,0) -- (0:1) -- (60:1) -- cycle;
							
								\begin{scope}[shift={(0:1)}]
									\draw[color=blue] (0,0) -- (0:1) -- (60:1) -- cycle;

									\begin{scope}[transparency group, opacity=0.2]
										\draw [line width=10, draw=blue, line cap=round]
										(120:1) -- (0,0) -- (0:1) -- +(300:1);
										\draw [line width=10, draw=blue, line cap=round]
										(0,0) -- (60:2);
										\draw [line width=10, draw=blue, line cap=round]
										(0:1) -- (60:1);
									\end{scope}
								
									\begin{scope}[shift={(0:1)}]
										\draw (0,0) -- (0:1) -- (300:1) -- cycle;
										\draw [line width=10, draw=red, opacity=0.2, line cap=round] (0:1) -- (300:1);
									\end{scope}
									
									\begin{scope}[shift={(60:1)}]
										\draw (0,0) -- (60:1) -- (120:1) -- cycle;
										\draw [line width=10, draw=red, opacity=0.2, line cap=round] (60:1) -- (120:1);
									\end{scope}
								\end{scope}
							\end{scope}

							\begin{scope}[shift={(120:1)}]
								\draw (0,0) -- (120:1) -- (180:1) -- cycle;
							
								\begin{scope}[shift={(120:1)}]
									\draw[color=blue] (0,0) -- (60:1) -- (120:1) -- cycle;

									\begin{scope}[transparency group, opacity=0.2]
										\draw [line width=10, draw=blue, line cap=round]
										(240:1) -- (60:1) -- +(0:1);
										\draw [line width=10, draw=blue, line cap=round]
										(0,0) -- (120:2);
										\draw [line width=10, draw=blue, line cap=round]
										(60:1) -- (120:1);
									\end{scope}
								
									\begin{scope}[shift={(60:1)}]
										\draw (0,0) -- (0:1) -- (60:1) -- cycle;
										\draw [line width=10, draw=red, opacity=0.2, line cap=round] (0:1) -- (60:1);
									\end{scope}
									
									\begin{scope}[shift={(120:1)}]
										\draw (0,0) -- (120:1) -- (180:1) -- cycle;
										\draw [line width=10, draw=red, opacity=0.2, line cap=round] (120:1) -- (180:1);
									\end{scope}
								\end{scope}
							\end{scope}
						\end{scope}
					\end{scope}
					
					\begin{scope}[shift={(120:1)}]
						\draw (0,0) -- (120:1) -- (180:1) -- cycle;
						\draw [line width=10, draw=red, opacity=0.2, line cap=round] (120:1) -- (180:1);
					\end{scope}
				\end{scope}
			\end{scope}

			\begin{scope}[shift={(180:1)}]
				\draw (0,0) -- (180:1) -- (240:1) -- cycle;
				
				\begin{scope}[shift={(180:1)}]
					\draw[color=blue] (0,0) -- (180:1) -- (240:1) -- cycle;
					
					\begin{scope}[transparency group, opacity=0.2]
						\draw [line width=10, draw=blue, line cap=round]
						(300:1) -- (0,0) -- (180:1) -- +(120:1);
						\draw [line width=10, draw=blue, line cap=round]
						(0,0) -- (240:2);
						\draw [line width=10, draw=blue, line cap=round]
						(180:1) -- (240:1);
					\end{scope}
					
					\begin{scope}[shift={(180:1)}]
						\draw (0,0) -- (120:1) -- (180:1) -- cycle;
						
						\begin{scope}[shift={(180:1)}]
							\draw[color=blue] (0,0) -- (120:1) -- (180:1) -- cycle;

							\begin{scope}[transparency group, opacity=0.2]
								\draw [line width=10, draw=green, line cap=round]
								(60:1) -- (0,0) -- (180:1) -- +(240:1);
								\draw [line width=10, draw=green, line cap=round] (0,0) -- (120:2);
								\draw [line width=10, draw=green, line cap=round] (120:1) -- (180:1);
							\end{scope}
							
							\begin{scope}[shift={(120:1)}]
								\draw (0,0) -- (60:1) -- (120:1) -- cycle;
							
								\begin{scope}[shift={(60:1)}]
									\draw[color=blue] (0,0) -- (60:1) -- (120:1) -- cycle;

									\begin{scope}[transparency group, opacity=0.2]
										\draw [line width=10, draw=blue, line cap=round]
										(180:1) -- (0,0) -- (60:1) -- +(0:1);
										\draw [line width=10, draw=blue, line cap=round]
										(0,0) -- (120:2);
										\draw [line width=10, draw=blue, line cap=round]
										(60:1) -- (120:1);
									\end{scope}
								
									\begin{scope}[shift={(60:1)}]
										\draw (0,0) -- (0:1) -- (60:1) -- cycle;
										\draw [line width=10, draw=red, opacity=0.2, line cap=round] (0:1) -- (60:1);
									\end{scope}
									
									\begin{scope}[shift={(120:1)}]
										\draw (0,0) -- (120:1) -- (180:1) -- cycle;
										\draw [line width=10, draw=red, opacity=0.2, line cap=round] (120:1) -- (180:1);
									\end{scope}
								\end{scope}
							\end{scope}

							\begin{scope}[shift={(180:1)}]
								\draw (0,0) -- (180:1) -- (240:1) -- cycle;
							
								\begin{scope}[shift={(180:1)}]
									\draw[color=blue] (0,0) -- (120:1) -- (180:1) -- cycle;

									\begin{scope}[transparency group, opacity=0.2]
										\draw [line width=10, draw=blue, line cap=round]
										(300:1) -- (120:1) -- +(60:1);
										\draw [line width=10, draw=blue, line cap=round]
										(0,0) -- (180:1) -- +(120:1);
										\draw [line width=10, draw=blue, line cap=round]
										(120:1) -- (180:1);
									\end{scope}
								
									\begin{scope}[shift={(120:1)}]
										\draw (0,0) -- (60:1) -- (120:1) -- cycle;
										\draw [line width=10, draw=red, opacity=0.2, line cap=round] (60:1) -- (120:1);
									\end{scope}
									
									\begin{scope}[shift={(180:1)}]
										\draw (0,0) -- (120:1) -- (180:1) -- cycle;
									\end{scope}
								\end{scope}
							\end{scope}
						\end{scope}
					\end{scope}
					
					\begin{scope}[shift={(240:1)}]
						\draw (0,0) -- (240:1) -- (300:1) -- cycle;
						\draw [line width=10, draw=red, opacity=0.2, line cap=round]
						(240:1) -- (300:1);
					\end{scope}
				\end{scope}
			\end{scope}
		\end{scope}

		\begin{scope}[shift={(-15,0)}]
			\draw[color=blue] (0,0) -- (120:0.8);
			\draw[color=blue] (0,0) -- (180:0.8);
		\end{scope}
		
		\end{tikzpicture}
		
		\caption{Supports of eigenfunctions of $M_\xi$ of eigenvalue $\frac{1}{4}$ for the Fabrykowski-Gupta group.}
		\label{fig:gupta-fabrykowski-fups-1}
	\end{figure}
\end{center}

%% file: img/gupta-fabrykowski-fups-2.tex
\begin{center}
	\begin{figure}[h]
		\begin{tikzpicture}[scale=0.895]
		\begin{scope}[shift={(0,0)}]
			\draw (0,0) -- (240:1) -- (-1, 0) -- cycle;
		\end{scope}
		
		\begin{scope}[shift={(-1,0)}]
			\draw[color=blue] (0,0) -- (120:1) -- (180:1) -- cycle;
			
			\begin{scope}[transparency group, opacity=0.2]
				\draw [line width=10, draw=red, line cap=round, rounded corners]
				(0,0) -- (0:1) -- (300:1) -- (120:2) -- +(0:1) -- (120:1);
			\end{scope}
			
			\begin{scope}[shift={(120:1)}]
				\draw (0,0) -- (60:1) -- (120:1) -- cycle;
			\end{scope}
			
			\begin{scope}[shift={(180:1)}]
				\draw (0,0) -- (120:1) -- (180:1) -- cycle;
			\end{scope}
		\end{scope}
		
		\begin{scope}[shift={(-3,0)}]
			\draw[color=blue] (0,0) -- (120:1) -- (180:1) -- cycle;
			
			\begin{scope}[transparency group, opacity=0.2]
				\draw [line width=10, draw=blue, line cap=round, rounded corners]
				(0:2) -- (0:3) -- +(240:1) -- (0:2) -- (0:1) -- +(120:1) -- (0,0) -- (0:1);
				\draw [line width=10, draw=blue, line cap=round, rounded corners, shift={(180:2)}]
				(120:1) -- +(60:1) -- (120:2) -- (300:1) -- +(60:1) -- (0,0);
				\draw [line width=10, draw=blue, line cap=round, rounded corners, shift={(120:1)}]
				(60:2) -- (60:3) -- +(300:1) -- (60:2) -- (0,0) -- (120:1) -- (60:1);
				\draw [line width=10, draw=blue, line cap=round, rounded corners]
				(0,0) -- (120:1) -- (180:1) -- (0,0);
			\end{scope}
			
			\begin{scope}[shift={(120:1)}]
				\draw (0,0) -- (60:1) -- (120:1) -- cycle;
				
				\begin{scope}[shift={(60:1)}]
					\draw[color=blue] (0,0) -- (60:1) -- (120:1) -- cycle;
					
					\begin{scope}[transparency group, opacity=0.2]
						\draw [line width=10, draw=red, line cap=round, rounded corners]
						(60:1) -- (60:2) -- +(300:1) -- (120:1) -- +(180:1) -- (120:2) -- (120:1);
					\end{scope}
					
					\begin{scope}[shift={(60:1)}]
						\draw (0,0) -- (0:1) -- (60:1) -- cycle;
					\end{scope}
					
					\begin{scope}[shift={(120:1)}]
						\draw (0,0) -- (120:1) -- (180:1) -- cycle;
					\end{scope}
				\end{scope}
			\end{scope}
			
			\begin{scope}[shift={(180:1)}]
				\draw (0,0) -- (180:1) -- (240:1) -- cycle;
				
				\begin{scope}[shift={(180:1)}]
					\draw[color=blue] (0,0) -- (120:1) -- (180:1) -- cycle;
				
					\begin{scope}[shift={(120:1)}]
						\draw (0,0) -- (60:1) -- (120:1) -- cycle;
					\end{scope}
					
					\begin{scope}[shift={(180:1)}]
					\draw (0,0) -- (120:1) -- (180:1) -- cycle;
					\end{scope}
				\end{scope}
			\end{scope}
		\end{scope}
		
		\begin{scope}[shift={(-7,0)}]
			\draw[color=blue] (0,0) -- (120:1) -- (180:1) -- cycle;
			
			\begin{scope}[transparency group, opacity=0.2]
				\draw [line width=10, draw=green, line cap=round, shift={(0:1)}, rounded corners] (60:1) -- +(120:1) -- (60:2) -- (0,0) -- (180:1) -- (120:1) --(0,0);
				\draw [line width=10, draw=green, line cap=round, shift={(120:1)}, shift={(60:1)}, rounded corners] (120:1) -- +(180:1) -- (120:2) -- (0,0) -- (240:1) -- (180:1) -- (0,0);
				\draw [line width=10, draw=green, line cap=round, shift={(180:2)}, rounded corners] (240:1) -- +(300:1) -- (240:2) -- (0,0) -- (300:1) -- +(60:1) -- (0,0);
				\draw [line width=10, draw=green, line cap=round, rounded corners] (0,0) -- (120:1) -- (180:1) -- (0,0);
			\end{scope}
			
			\begin{scope}[shift={(120:1)}]
				\draw (0,0) -- (60:1) -- (120:1) -- cycle;
				
				\begin{scope}[shift={(60:1)}]
					\draw[color=blue] (0,0) -- (60:1) -- (120:1) -- cycle;
					
					\begin{scope}[shift={(60:1)}]
						\draw (0,0) -- (0:1) -- (60:1) -- cycle;
						
						\begin{scope}[shift={(60:1)}]
							\draw[color=blue] (0,0) -- (60:1) -- (120:1) -- cycle;

							\begin{scope}[transparency group, opacity=0.2]
								\draw [line width=10, draw=blue, line cap=round, shift={(240:1)}, rounded corners] (180:1) -- +(120:1) -- (180:2) -- (0:1) -- (60:1) -- (0,0);
								\draw [line width=10, draw=blue, line cap=round, shift={(60:1)}, rounded corners] (0:2) -- (0:3) -- +(240:1) -- (0:2) -- (0,0) -- (60:1) -- (0:1);
								\draw [line width=10, draw=blue, line cap=round, shift={(120:2)}, rounded corners] (60:1) -- +(0:1) -- (60:2) -- (240:1) -- +(0:1) -- (0,0);
								\draw [line width=10, draw=blue, line cap=round, rounded corners] (0,0) -- (60:1) -- (120:1) -- (0,0);
							\end{scope}
							
							\begin{scope}[shift={(60:1)}]
								\draw (0,0) -- (0:1) -- (60:1) -- cycle;
							
								\begin{scope}[shift={(0:1)}]
									\draw[color=blue] (0,0) -- (0:1) -- (60:1) -- cycle;

									\begin{scope}[transparency group, opacity=0.2]
										\draw [line width=10, draw=red, line cap=round, rounded corners]
										(0:1) -- (0:2) -- +(240:1) -- (60:1) -- +(120:1) -- (60:2) -- (60:1);
									\end{scope}
								
									\begin{scope}[shift={(0:1)}]
										\draw (0,0) -- (0:1) -- (300:1) -- cycle;
									\end{scope}
									
									\begin{scope}[shift={(60:1)}]
										\draw (0,0) -- (60:1) -- (120:1) -- cycle;
									\end{scope}
								\end{scope}
							\end{scope}

							\begin{scope}[shift={(120:1)}]
								\draw (0,0) -- (120:1) -- (180:1) -- cycle;
							
								\begin{scope}[shift={(120:1)}]
									\draw[color=blue] (0,0) -- (60:1) -- (120:1) -- cycle;

									\begin{scope}[transparency group, opacity=0.2]
										\draw [line width=10, draw=red, line cap=round, rounded corners]
										(60:1) -- (60:2) -- +(300:1) -- (120:1) -- +(180:1) -- (120:2) -- (120:1);
									\end{scope}
								
									\begin{scope}[shift={(60:1)}]
										\draw (0,0) -- (0:1) -- (60:1) -- cycle;
									\end{scope}
									
									\begin{scope}[shift={(120:1)}]
										\draw (0,0) -- (120:1) -- (180:1) -- cycle;
									\end{scope}
								\end{scope}
							\end{scope}
						\end{scope}
					\end{scope}
					
					\begin{scope}[shift={(120:1)}]
						\draw (0,0) -- (120:1) -- (180:1) -- cycle;
					\end{scope}
				\end{scope}
			\end{scope}

			\begin{scope}[shift={(180:1)}]
				\draw (0,0) -- (180:1) -- (240:1) -- cycle;
				
				\begin{scope}[shift={(180:1)}]
					\draw[color=blue] (0,0) -- (180:1) -- (240:1) -- cycle;
					
					\begin{scope}[shift={(180:1)}]
						\draw (0,0) -- (120:1) -- (180:1) -- cycle;
						
						\begin{scope}[shift={(180:1)}]
							\draw[color=blue] (0,0) -- (120:1) -- (180:1) -- cycle;

							\begin{scope}[transparency group, opacity=0.2]
								\draw [line width=10, draw=blue, line cap=round, shift={(0:1)}, rounded corners] (300:1) -- +(240:1) -- (300:2) -- (120:1) -- (180:1) -- (0,0);
								\draw [line width=10, draw=blue, line cap=round, shift={(120:1)}, rounded corners] (60:2) -- (60:3) -- +(300:1) -- (60:2) -- (0,0) -- (120:1) -- (60:1);
								\draw [line width=10, draw=blue, line cap=round, shift={(180:2)}, rounded corners] (120:1) -- +(60:1) -- (120:2) -- (300:1) -- +(60:1) -- (0,0);
								\draw [line width=10, draw=blue, line cap=round, rounded corners] (0,0) -- (120:1) -- (180:1) -- (0,0);
							\end{scope}
							
							\begin{scope}[shift={(120:1)}]
								\draw (0,0) -- (60:1) -- (120:1) -- cycle;
							
								\begin{scope}[shift={(60:1)}]
									\draw[color=blue] (0,0) -- (60:1) -- (120:1) -- cycle;

									\begin{scope}[transparency group, opacity=0.2]
										\draw [line width=10, draw=red, line cap=round, rounded corners]
										(60:1) -- (60:2) -- +(300:1) -- (120:1) -- +(180:1) -- (120:2) -- (120:1);
									\end{scope}
								
									\begin{scope}[shift={(60:1)}]
										\draw (0,0) -- (0:1) -- (60:1) -- cycle;
									\end{scope}
									
									\begin{scope}[shift={(120:1)}]
										\draw (0,0) -- (120:1) -- (180:1) -- cycle;
									\end{scope}
								\end{scope}
							\end{scope}

							\begin{scope}[shift={(180:1)}]
								\draw (0,0) -- (180:1) -- (240:1) -- cycle;
							
								\begin{scope}[shift={(180:1)}]
									\draw[color=blue] (0,0) -- (120:1) -- (180:1) -- cycle;
								
									\begin{scope}[shift={(120:1)}]
										\draw (0,0) -- (60:1) -- (120:1) -- cycle;
									\end{scope}
									
									\begin{scope}[shift={(180:1)}]
										\draw (0,0) -- (120:1) -- (180:1) -- cycle;
									\end{scope}
								\end{scope}
							\end{scope}
						\end{scope}
					\end{scope}
					
					\begin{scope}[shift={(240:1)}]
						\draw (0,0) -- (240:1) -- (300:1) -- cycle;
					\end{scope}
				\end{scope}
			\end{scope}
		\end{scope}

		\begin{scope}[shift={(-15,0)}]
			\draw[color=blue] (0,0) -- (120:0.8);
			\draw[color=blue] (0,0) -- (180:0.8);
		\end{scope}
		
		\end{tikzpicture}
		
		\caption{Supports of eigenfunctions of $M_\xi$ of eigenvalue $\frac{1}{4}(1 \pm \sqrt{6})$ for the Fabrykowski-Gupta group.}
		\label{fig:gupta-fabrykowski-fups-2}
	\end{figure}
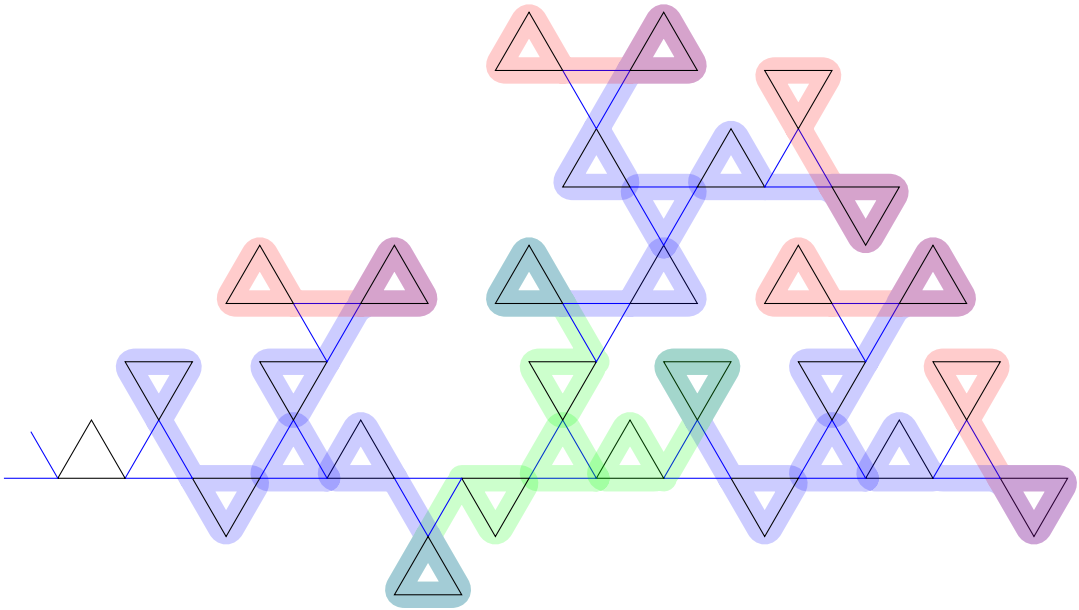
\end{center}